\documentclass[letterpaper,final]{siamltex}
\usepackage{amsfonts,amsmath,amssymb}
\usepackage{graphicx}
\usepackage[ruled,titlenotnumbered,nofillcomment]{algorithm2e}
\usepackage{multirow,stmaryrd}
\usepackage[usenames]{color}
\usepackage{epsfig}
\usepackage{epsf}
\usepackage{subfigure}
\newcommand{\beq}{\begin{equation}}
\newcommand{\eeq}{\end{equation}}
\newcommand{\beqs}{\begin{equation*}}
\newcommand{\eeqs}{\end{equation*}}

\newcommand{\bb}{{\bf b}}
\newcommand{\bx}{{\bf x}}
\newcommand{\by}{{\bf y}}

\newcommand{\bd}{{\bf d}}
\newcommand{\bv}{{\bf v}}

\newcommand{\bff}{{\bf f}}
\newcommand{\bg}{{\bf g}}
\newcommand{\bu}{{\bf u}}

\newcommand{\bz}{{\bf z}}
\newcommand{\bp}{{\bf p}}
\newcommand{\br}{{\bf r}}

\newcommand{\bw}{{\bf w}}

\newcommand{\Div}{\nabla\cdot}

\newcommand{\BC}{\begin{center}}
\newcommand{\EC}{\end{center}}
\newcommand{\bdm}{\begin{displaymath}}
\newcommand{\edm}{\end{displaymath}}
\newcommand{\CF}{{\cal F}}

\newcommand{\dia}[1]{\text{diag}(#1)}

\newcommand{\argmin}{\mathop{\mbox{argmin}}}

\newcommand{\barr}{\begin{array}}
\newcommand{\earr}{\end{array}}

\newcommand{\beqas}{\begin{eqnarray*}}
\newcommand{\eeqas}{\end{eqnarray*}}

\title{Steepest Descent Preconditioning for Nonlinear GMRES Optimization}

\author{
H. De Sterck\footnotemark[1] \footnotemark[4] 
}

\begin{document}
%%%%%%%%%%%%%%%%%%%%%%%%%%%%%%%%%%%%%%%%%%%%%%%%%%%%%%%%%%
%%%%%%%%%%%%%%%%%%%%%%%%%%%%%%%%%%%%%%%%%%%%%%%%%%%%%%%%%%
%%%%%%%%%%%%%%%%%%%%%%%%%%%%%%%%%%%%%%%%%%%%%%%%%%%%%%%%%%
%\psdraft
%
%\centerline{\bf *** draft revision, \today ***}
%\centerline{\sc \tiny \bf preprint}
%\centerline{\sc \tiny \bf Submitted to , arXiv:1106.4426.}
%\centerline{\sc \tiny \bf Submitted to .}
%\centerline{\sc \tiny \bf submitted, , arXiv:1106.4426}
%\centerline{\sc \tiny \bf submitted, }
\maketitle
\renewcommand{\thefootnote}{\fnsymbol{footnote}}
\footnotetext[1]{Department of Applied Mathematics, University of Waterloo,
Waterloo, Ontario, Canada}
\footnotetext[4]{hdesterck@uwaterloo.ca}
%\centerline{\today \quad (*** draft, not for circulation ***)}
%\centerline{\today}
%\renewcommand{\thefootnote}{\arabic{footnote}}
%%%%%%%%%%%%%%%%%%%%%%%%%%%%%%%%%%%%%%%%%%%%%%%%%%%%%%%%%%
%%%%%%%%%%%%%%%%%%%%%%%%%%%%%%%%%%%%%%%%%%%%%%%%%%%%%%%%%%
\begin{abstract}
Steepest descent preconditioning is considered for the recently proposed nonlinear generalized minimal
residual (N-GMRES) optimization algorithm for unconstrained nonlinear optimization. 
Two steepest descent preconditioning variants are proposed.
The first employs a line search, while the second employs a predefined small step.
A simple global convergence proof is provided for the N-GMRES optimization algorithm with the first steepest descent
preconditioner (with line search), under mild standard conditions on the objective function and the line
search processes. Steepest descent preconditioning for N-GMRES optimization is also motivated by relating it to
standard non-preconditioned GMRES for linear systems in the case of a standard quadratic optimization
problem with symmetric positive definite operator.
Numerical tests on a variety of model problems show that the N-GMRES optimization algorithm is able to very
significantly accelerate convergence of stand-alone steepest descent optimization.
Moreover, performance of steepest-descent
preconditioned N-GMRES is shown to be competitive with standard nonlinear
conjugate gradient and limited-memory Broyden-Fletcher-Goldfarb-Shanno methods for the model problems considered.
These results serve to theoretically and numerically establish steepest-descent preconditioned N-GMRES as a general optimization method for unconstrained nonlinear optimization, with performance that appears
promising compared to established techniques.
In addition, it is argued that the real potential of the N-GMRES optimization framework lies in the fact
that it can make use of problem-dependent nonlinear preconditioners that are more powerful than steepest descent
(or, equivalently, N-GMRES can be used as a simple wrapper around any other iterative optimization process to seek acceleration of that process), and this potential is illustrated with a further application example.
\end{abstract}
%%%%%%%%%%%%%%%%%%%%%%%%%%%%%%%%%%%%%%%%%%%%%%%%%%%%%%%%%%
\pagestyle{myheadings}
\thispagestyle{plain}
\markboth{H. De Sterck
}{Steepest Descent Preconditioning for N-GMRES Optimization}
%%%%%%%%%%%%%%%%%%%%%%%%%%%%%%%%%%%%%%%%%%%%%%%%%%%%%%%%%%
% sections
%%%%%%%%%%%%%%%%%%%%%%%%%%%%%%%%%%%%%%%%%%%%%%%%%%%%%%%%%%

%%%%%%%%%%%%%%%%%%%%%%%%%%%%%%%%%%%%%%%%%%%%%%%%%%%%%%%%%%
%%%%%%%%%%%%%%%%%%%%%%%%%%%%%%%%%%%%%%%%%%%%%%%%%%%%%%%%%%
\section{Introduction}
%%%%%%%%%%%%%%%%%%%%%%%%%%%%%%%%%%%%%%%%%%%%%%%%%%%%%%%%%%
%%%%%%%%%%%%%%%%%%%%%%%%%%%%%%%%%%%%%%%%%%%%%%%%%%%%%%%%%%
In recent work on canonical tensor approximation \cite{NGMRES}, we have proposed an algorithm that accelerates convergence of the alternating least squares (ALS) optimization method for the canonical tensor approximation problem considered there. The algorithm proceeds by linearly recombining previous iterates in a way that approximately minimizes the residual (the gradient of the objective function), using a nonlinear generalized minimal residual (GMRES) approach. The recombination step is followed by a line search step for globalization, and the resulting three-step non-linear GMRES (N-GMRES) optimization algorithm is shown in \cite{NGMRES} to significantly speed up the convergence of ALS for the canonical tensor approximation problem considered.

As explained in \cite{NGMRES} (which we refer to as Paper I in what follows), for the tensor approximation problem considered there, ALS can also be interpreted as a preconditioner for the N-GMRES optimization algorithm.
The question then arises what other types of preconditioners can be considered for the N-GMRES optimization algorithm proposed in Paper I, and whether there are universal preconditioning approaches that can make the N-GMRES optimization algorithm applicable to nonlinear optimization problems more generally.
In the present paper, we propose such a universal preconditioning approach for the N-GMRES optimization algorithm proposed in Paper I, namely, steepest descent preconditioning. 
We explain how updates in the steepest descent direction can indeed naturally be used as a preconditioning process for the N-GMRES optimization algorithm.
%, thus generalizing the applicability of the N-GMRES optimization method proposed in Paper I to a broad class of smooth unconstrained nonlinear optimization problems. 
In fact, we show that steepest descent preconditioning can be seen as the most basic preconditioning process for the N-GMRES optmization method, in the sense that applying N-GMRES to a quadratic objective function with symmetric positive definite (SPD) operator, corresponds mathematically to applying standard non-preconditioned GMRES for linear systems to the linear system corresponding to the quadratic objective function.
We propose two variants of steepest descent preconditioning, one with line search and one with a predefined small step. We 
give a simple global convergence proof for the N-GMRES optimization algorithm with our first proposed variant of steepest descent preconditioning (with line search), under standard mild conditions on the objective function and for line searches satisfying the Wolfe conditions. The second preconditioning approach, without line search, is of interest because it is more efficient in numerical tests, but there is no convergence guarantee. Numerical results are employed for a variety of test problems demonstrating that N-GMRES optimization can significantly speed up stand-alone steepest descent optimization. We also compare steepest-descent preconditioned N-GMRES with a standard nonlinear conjugate gradient (N-CG) method for all our test problems, and with a standard limited-memory Broyden-Fletcher-Goldfarb-Shanno (L-BFGS) method.
%Note that, among established methods for nonlinear optimization, we choose to compare with N-CG because it is, like our N-GMRES algorithm, a generalization to nonlinear optimization of a Krylov method for linear equations. %Our results will show that steepest-descent preconditioned N-GMRES is in many cases competitive with N-CG.

We consider the following unconstrained nonlinear optimization problem with associated first-order optimality equations:\\

%\vspace{.1cm}
\noindent
{\sc optimization problem I:}
%------------------------------------------------------------------
\begin{align}
\text{find $\bu^*$ that minimizes }f(\bu).
\label{eq:fu}
\end{align}
%------------------------------------------------------------------

\noindent
{\sc first-order optimality equations I:}
%------------------------------------------------------------------
\begin{align}
\nabla f(\bu)=\bg(\bu)=0.
\label{eq:gu}
\end{align}
%------------------------------------------------------------------

The N-GMRES optimization algorithm proposed in Paper I for accelerating ALS for canonical tensor approximation consists of three steps that can be summarized as follows. (Fig.~\ref{fig:N-GMRES} gives a schematic representation of the algorithm, and it is described in pseudo-code in Algorithm \ref{alg:N-GMRES}.)
In the first step, a preliminary new iterate $\bar{\bu}_{i+1}$ is generated from the last iterate $\bu_i$ using a one-step iterative update process $M(.)$, which can be interpreted as a preconditioning process (see Paper I and below). ALS preconditioning is used for $M(.)$ in Paper I.
In the second step, an accelerated iterate $\hat{\bu}_{i+1}$ is obtained by linearly recombining previous iterates in a window of size $w$, $(\bu_{i-w+1},\ldots,\bu_{i})$, using a nonlinear GMRES approach. (The details of this step will be recalled in Section \ref{sec:Steepest} below.)
In the third step, a line search is performed that minimizes objective function $\bff(\bu)$ on a half line starting at preliminary iterate $\bar{\bu}_{i+1}$, which was generated in Step I, and connecting it with accelerated iterate $\hat{\bu}_{i+1}$, which was generated in Step II, to obtain the new iterate $\bu_{i+1}$.

The second step in the N-GMRES optimization algorithm (Step II in Algorithm \ref{alg:N-GMRES}) uses the nonlinear extension of GMRES for solving nonlinear systems of equations that was proposed by Washio and Oosterlee in \cite{WashioNGMRES-ETNA} in the context of nonlinear partial differential equation (PDE) systems (see also \cite{OosterleeNGMRES-SISC} and \cite{WashioNGMRES-ETNA} for further applications to PDE systems). It is a nonlinear extension of the celebrated GMRES method for iteratively solving systems of linear equations \cite{SaadGMRES,SaadBook}. Washio and Oosterlee's nonlinear extension is related to Flexible GMRES as described in \cite{SaadFlexible}, and is also related to the reduced rank extrapolation method \cite{RRE}. An early description of this type of nonlinear iterate acceleration ideas for solving nonlinear equation systems appears in so-called Anderson mixing, see, e.g., \cite{SaadAnderson,Walker}. More recent applications of these ideas to nonlinear equation systems and fixed-point problems are discussed in \cite{SaadAnderson,Walker}. 
In Paper I we formulated a nonlinear GMRES optimization algorithm for canonical tensor decomposition that uses this type of acceleration as one of its steps, combined with an ALS preconditioning step and a line search for globalization. 
The type of nonlinear iterate acceleration in Step II of Algorithm \ref{alg:N-GMRES} has thus been considered several times before in the context of solving nonlinear systems of equations, but we believe that its combination with a line search to obtain a general preconditioned nonlinear optimization method as in Algorithm \ref{alg:N-GMRES} (see Paper I) is new in the optimization context.
In the present paper we show how this N-GMRES optimization approach can be applied to a broad class of sufficiently smooth nonlinear optimization problems by using steepest descent preconditioning. We establish theoretical convergence properties for this approach and demonstrate its effectiveness in numerical tests.\\

%%%%%%%%%%%%%%%%%%%%%%%%%%%%%%%%%%%%%%%%%%%%%%%%%
%%%%%%%%%%%%%%%%%%%%%%%%%%%%%%%%%%%%%%%%%%%%%%%%%
%\linesnumberedhidden
\begin{algorithm}[H]
\dontprintsemicolon

{\bf Input:} $w$ initial iterates $\bu_0, \ldots,\bu_{w-1}$.\;
\ \\
$i=w-1$\;
\Repeat{\text{convergence criterion satisfied}}{
{\sc Step I:} {\em (generate preliminary iterate by one-step update process $M(.)$)}\;
$\qquad \bar{\bu}_{i+1}=M(\bu_{i})$\;
{\sc Step II:}  {\em (generate accelerated iterate by nonlinear GMRES step)}\;
$\qquad \hat{\bu}_{i+1}=$gmres$(\bu_{i-w+1},\ldots,\bu_{i};\bar{\bu}_{i+1})$\;
{\sc Step III:}  {\em (generate new iterate by line search process)}\;
\qquad {\bf if } $\hat{\bu}_{i+1}-\bar{\bu}_{i+1}$ {\em is a descent direction}\;
\qquad \qquad $\bu_{i+1}=$linesearch$(\bar{\bu}_{i+1}+\beta(\hat{\bu}_{i+1}-\bar{\bu}_{i+1}))$\;
\qquad {\bf else}\;
\qquad \qquad $\bu_{i+1}=\bar{\bu}_{i+1}$\;
\qquad {\bf end}\;
$i=i+1$\;
}\;
\ \\
\caption{N-GMRES optimization algorithm (window size $w$)}
\label{alg:N-GMRES}
\end{algorithm}
%%%%%%%%%%%%%%%%%%%%%%%%%%%%%%%%%%%%%%%%%%%%%%%%%
%%%%%%%%%%%%%%%%%%%%%%%%%%%%%%%%%%%%%%%%%%%%%%%%%
(Note that the $w$ initial iterates required in Algorithm \ref{alg:N-GMRES} can naturally be generated by applying the algorithm with a window size that gradually increases from one up to $w$, starting from a single initial guess. Also, as in \cite{NGMRES}, we perform a restart and reset the window size back to 1 whenever $\hat{\bu}_{i+1}-\bar{\bu}_{i+1}$ is not a descent direction.)\\

%---------------------------------------------------------------------------------------------------------------------------------
\begin{figure}[!htbp]
  \centering
  \scalebox{1.2}{
    \includegraphics{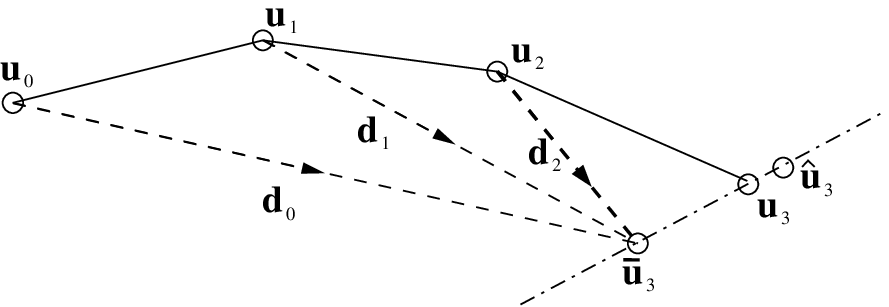}
  }
   \caption{Schematic representation of one iteration of the N-GMRES optimization algorithm (from \cite{NGMRES}). Given previous iterations $\bu_0$, $\bu_1$ and $\bu_2$, new iterate $\bu_3$ is generated as follows. In Step I, preliminary iterate $\bar{\bu}_3$ is generated by the one-step update process $M(.)$: $\bar{\bu}_3=M(\bu_2)$. In Step II, the nonlinear GMRES step, accelerated iterate $\hat{\bu}_3$ is obtained by determining the coefficients $\alpha_j$ in $\hat{\bu}_3=\bar{\bu}_3+\alpha_0 \bd_0+\alpha_1 \bd_1+\alpha_2 \bd_2$ such that the gradient of the objective function in $\hat{\bu}_3$ is approximately minimized. In Step III, the new iterate, $\bu_3$, is finally generated by a line search that minimizes the objective function $f(\bar{\bu}_{3}+\beta(\hat{\bu}_{3}-\bar{\bu}_{3}))$.}
   \label{fig:N-GMRES}
\end{figure}    
%---------------------------------------------------------------------------------------------------------------------------------

The rest of this paper is structured as follows. In Section \ref{sec:Steepest} we propose two types of steepest descent preconditioners for N-GMRES Optimization Algorithm \ref{alg:N-GMRES}.
We briefly recall the details of the nonlinear GMRES optimization step, give a motivation and interpretation for steepest descent preconditioning that relate it to non-preconditioned GMRES for SPD linear systems, and give a simple proof for global convergence of the N-GMRES optimization algorithm using steepest descent preconditioning with line search.
In Section \ref{sec:numerics} we present extensive numerical results for N-GMRES optimization with the two proposed steepest descent preconditioners, applied to a variety of nonlinear optimization problems, and compare with stand-alone steepest descent, N-CG and L-BFGS.
Finally, Section \ref{sec:conc} concludes.

%%%%%%%%%%%%%%%%%%%%%%%%%%%%%%%%%%%%%%%%%%%%%%%%%%%%%%%%%%
%%%%%%%%%%%%%%%%%%%%%%%%%%%%%%%%%%%%%%%%%%%%%%%%%%%%%%%%%%
\section{Steepest Descent Preconditioning for N-GMRES Optimization}
\label{sec:Steepest}
%%%%%%%%%%%%%%%%%%%%%%%%%%%%%%%%%%%%%%%%%%%%%%%%%%%%%%%%%%
%%%%%%%%%%%%%%%%%%%%%%%%%%%%%%%%%%%%%%%%%%%%%%%%%%%%%%%%%%
In this section, we first propose two variants of steepest descent preconditioning. We then briefly recall the details of the nonlinear GMRES recombination step (Step II in Algorithm \ref{alg:N-GMRES}), and relate N-GMRES optimization to
standard non-preconditioned GMRES for linear systems in the case of a simple quadratic optimization
problem with SPD operator. Finally, we give a simple global convergence proof for the N-GMRES optimization algorithm using steepest descent preconditioning with line search.

%*****************************************************************
\subsection{Steepest Descent Preconditioning Process}
\label{subsec:SteepestProcess}
%*****************************************************************
We propose a general steepest descent preconditioning process for Step I of N-GMRES Optimization Algorithm 
\ref{alg:N-GMRES} with the following two variants:\\

\noindent
{\sc Steepest Descent Preconditioning Process:}
%------------------------------------------------------------------
%\begin{align}
\begin{alignat}{3}
&\quad \bar{\bu}_{i+1}=\bu_i-\beta \, \frac{\nabla f(\bu_i)}{\|\nabla f(\bu_i)\|} &&\quad \text{with} \nonumber \\
%\label{eq:steepest}
%\end{align}
%------------------------------------------------------------------
%------------------------------------------------------------------
%\begin{alignat}{3}
&\text{\sc option A:}&\beta&=\beta_{sdls}, \label{eq:steepestA}\\
&\text{\sc option B:}&\beta&=\beta_{sd}=\min(\, \delta \, , \, \|\nabla f(\bu_i)\| \, ).
\label{eq:steepestB}
\end{alignat}
%------------------------------------------------------------------
For Option A, $\beta_{sdls}$ is the step length obtained by a line search procedure. For definiteness, we consider a line search procedure that satisfies the Wolfe conditions (see below). We refer to the steepest descent preconditioning process with line search (\ref{eq:steepestA}) as the {\em sdls} preconditioner.
For Option B, we predefine the step $\beta_{sd}$ as the minimum of a small positive constant $\delta$, and the norm of the gradient. In the numerical results to be presented further on in the paper, we use $\delta=10^{-4}$, except where noted. We refer to the steepest descent preconditioning process with predefined step $\beta_{sd}$ (\ref{eq:steepestB}) as the {\em sd} preconditioner. These two Options are quite different, and some discussion is in order.

Preconditioning process A can be employed as a stand-alone optimization method (it can converge by itself), and N-GMRES can be considered as a wrapper that accelerates this stand-alone process. We will show below that N-GMRES with preconditioning process A has strong convergence properties, but it may be expensive because the line search may require a significant number of function and gradient ($f/g$) evaluations.
However, the situation is very different for preconditioning process B. Here, no additional $f/g$ evaluations are required, but convergence appears questionable. It is clear that preconditioning process B cannot be used as a stand-alone optimization algorithm; in most cases it would not converge. It can, however, still be used as a preconditioning process for N-GMRES. As is well-known and will be further illustrated below, preconditioners used by GMRES for linear systems do not need to be convergent by themselves, and this suggests that it may be interesting to consider this for N-GMRES optimization as well. As will be motivated further below, the role of the N-GMRES preconditioning process is to provide new `useful' directions for the nonlinear generalization of the Krylov space, and the iteration can be driven to convergence by the N-GMRES minimization, even if the preconditioner is not convergent by itself. However, for this to happen in the three-step N-GMRES optimization algorithm with preconditioning process B, it is required that $\bar{\bu}_{i+1}$ eventually approaches $\bu_i$ and the step length $\beta_{sd}$ approaches 0. For this reason, we select $\beta_{sd}= \|\nabla f(\bu_i)\|$ as soon as $\|\nabla f(\bu_i)\| \le \delta$. The initial step length $\beta_{sd}$ is chosen to be not larger than a small constant because the linear case (see below) suggests that a small step is sufficient to provide a new direction for the Krylov space, and because the minimization of the residual is based on a linearization argument (see also below), and small steps tend to lead to small linearization errors.

%*****************************************************************
\subsection{N-GMRES Recombination Step}
\label{subsec:N-GMRESStep}
%*****************************************************************
Before relating steepest-descent preconditioned N-GMRES to non-preconditioned GMRES for linear systems, we first recall from \cite{NGMRES} some details of the N-GMRES recombination step, Step II in Algorithm \ref{alg:N-GMRES}.
In this step, we find an accelerated iterate $\hat{\bu}_{i+1}$ that is obtained by recombining previous iterates as follows:
%------------------------------------------------------------------
\begin{align}
\hat{\bu}_{i+1}=\bar{\bu}_{i+1}+\sum_{j=0}^{i} \, \alpha_j \, (\bar{\bu}_{i+1}-\bu_j).
\label{eq:accel}
\end{align}
%------------------------------------------------------------------
The unknown coefficients $\alpha_j$ are determined by the N-GMRES algorithm in such a way that the two-norm of the gradient of the objective function evaluated at the accelerated iterate is small. In general, $\bg(.)$ is a nonlinear function of the $\alpha_j$, and linearization is used to allow for inexpensive computation of coefficients $\alpha_j$ that may approximately minimize $\|\bg(\hat{\bu}_{i+1})\|_2$. Using the following approximations
%------------------------------------------------------------------
\begin{align}
\bg(\hat{\bu}_{i+1})&\approx \bg(\bar{\bu}_{i+1})+\sum_{j=0}^{i} \, \left. \frac{\partial \bg}{\partial \bu} \right|_{\bar{\bu}_{i+1}} \, \alpha_j \, (\bar{\bu}_{i+1}-\bu_{j}) \nonumber\\
  & \approx \bg(\bar{\bu}_{i+1})+\sum_{j=0}^{i} \, \alpha_j \, (\bg(\bar{\bu}_{i+1})-\bg(\bu_{j}))
  \label{eq:linearize}
\end{align}
%------------------------------------------------------------------
one arrives at minimization problem
%\centerline{{find $(\alpha_0, \ldots, \alpha_i)$ that minimize }}
%------------------------------------------------------------------
%\begin{align}
% \| \bg(\bu_{i+1})+\sum_{j=0}^{i} \, \alpha_j \, (\bg(\bu_{i+1})-\bg(\bu_{j})) \|_2.
%\label{eq:minAlpha}
%\end{align}
%------------------------------------------------------------------
%------------------------------------------------------------------
\begin{gather}
\text{find coefficients $(\alpha_0, \ldots, \alpha_i)$ that minimize } \nonumber\\
 \| \bg(\bar{\bu}_{i+1})+\sum_{j=0}^{i} \, \alpha_j \, (\bg(\bar{\bu}_{i+1})-\bg(\bu_{j})) \|_2.
\label{eq:minAlpha}
\end{gather}
%------------------------------------------------------------------
This is a standard least-squares problem that can be solved, for example, by using the normal equations, as explained in \cite{WashioNGMRES-ETNA,NGMRES}. (In this paper, we solve the least-squares problem as described in \cite{NGMRES}.)

In a windowed implementation with window size $w$, the memory cost incurred by N-GMRES acceleration is the storage of $w$ previous approximations and residuals. The dominant parts of the CPU cost for each acceleration step are the cost of building and solving the least-squares system (which can be done in approximately $2 n w$ flops if the normal equations are used and some previous inner products are stored, see \cite{WashioNGMRES-ETNA}), and $n w$ flops to compute the accelerated iterate. For problems with expensive objective functions, this cost is often negligible compared to the cost of the $f/g$ evaluations in the line searches \cite{NGMRES}.

%*****************************************************************
\subsection{Motivation and Interpretation for Steepest Descent Preconditioning}
\label{subsec:Motiv}
%*****************************************************************
Consider a standard quadratic minimization problem with objective function
%------------------------------------------------------------------
\begin{align}
f(\bu)=\frac{1}{2} \, \bu^T  A \bu  -  \bb^T  \bu,
\label{eq:fsimple}
\end{align}
%------------------------------------------------------------------
where $A$ is SPD. It is well-known that its unique minimizer satisfies $A \bu=\bb$.
Now consider applying the N-GMRES optimization algorithm with steepest descent preconditioner to the quadratic minimization problem. The gradient of $f$ at approximation $\bu_i$ is given by
%------------------------------------------------------------------
\begin{align}
\nabla f(\bu_i)=A\bu_i-b=-\br_i \quad \text{with} \quad \br_i=b-A \bu_i,
\label{eq:gradsimple}
\end{align}
%------------------------------------------------------------------
where $\br_i$ is defined as the residual of the linear system $A \bu=\bb$ in $\bu_i$.
N-GMRES steepest descent preconditioner (\ref{eq:steepestA})-(\ref{eq:steepestB}) then reduces to the form
%------------------------------------------------------------------
\begin{align}
\bar{\bu}_{i+1}=\bu_i+\beta \, \frac{\br_i}{\|\br_i\|},
\label{eq:precondsimple}
\end{align}
%------------------------------------------------------------------
and it can easily be shown that this corresponds to the stationary iterative method that generates the Krylov space in non-preconditioned linear GMRES applied to $A \bu=\bb$. We now briefly show this because it provides further insight (recalling parts of the discussion in \cite{WashioNGMRES-ETNA,NGMRES}). 

We first explain how preconditioned GMRES for $A \bu=\bb$ works.
Consider so-called stationary iterative methods for $A \bu=\bb$ of the following form:
%------------------------------------------------------------------
\begin{align}
\bu_{i+1}=\bu_i+M^{-1}\, \br_i.
\label{eq:stat}
\end{align}
%------------------------------------------------------------------
Here, matrix $M$ is an approximation of $A$ that has an easily computable inverse, i.e., $M^{-1}\approx A^{-1}$.
For example, $M$ can be chosen to correspond to Gauss-Seidel or Jacobi iteration, or to a multigrid cycle \cite{WashioNGMRES-ETNA}.

Consider a sequence of iterates $\bu_0, \ldots,\bu_i$ generated by update formula (\ref{eq:stat}), starting from some initial guess $\bu_0$. Note that the residuals of these iterates are related as $\br_i=\bb-A\,\bu_i=(I-A M^{-1})\,\br_{i-1}=(I-A M^{-1})^i\,\br_0.$ 
This motivates the definition of the following vector spaces:
%------------------------------------------------------------------
\begin{align}
V_{1,i+1}&=\mathop{span}\{ \br_0,\ldots,\br_i\}, \nonumber\\
V_{2,i+1}&=\mathop{span}\{ \br_0, A M^{-1} \, \br_0, (A M^{-1})^2 \, \br_0\}, \ldots, (A M^{-1})^i \, \br_0\} \nonumber\\
  &=K_{i+1}(A M^{-1},\br_0), \nonumber\\
V_{3,i+1}&=\mathop{span}\{ M \, (\bu_{i+1}-\bu_0),  M \, (\bu_{i+1}-\bu_1),\ldots,  M \, (\bu_{i+1}-\bu_i) \}.\nonumber
\end{align}
%------------------------------------------------------------------
Vector space $V_{2,i+1}$ is the so-called Krylov space $K_{i+1}(A M^{-1},\br_0)$ of order $i+1$, generated by matrix $A M^{-1}$ and vector $\br_0$.
It is easy to show that these vector spaces are equal (see, e.g., \cite{WashioNGMRES-ETNA,NGMRES}).

Expression (\ref{eq:stat}) shows that $M \, (\bu_{i+1}-\bu_i) \in K_{i+1}(A M^{-1},\br_0)$. The GMRES procedure can be seen as a way to accelerate stationary iterative method (\ref{eq:stat}), by recombining iterates (or, equivalently, by reusing residuals). In particular, we seek a better approximation $\hat{\bu}_{i+1}$, with $M \, (\hat{\bu}_{i+1}-\bu_i)$ in the Krylov space $K_{i+1}(A M^{-1},\br_0)$, such that $\hat{\br}_{i+1}=\bb-A\,\hat{\bu}_{i+1}$ has minimal two-norm.
In other words, we seek optimal coefficients $\beta_j$ in
%------------------------------------------------------------------
\begin{align*}
M \, (\hat{\bu}_{i+1}-\bu_i) &= \sum_{j=0}^{i} \, \beta_j \, M \, (\bu_{i+1}-\bu_j),\\
%\label{eq:resid}
\end{align*}
%------------------------------------------------------------------
and it is easy to show that this corresponds to seeking optimal coefficients $\alpha_j$ in
%------------------------------------------------------------------
\begin{align}
\hat{\bu}_{i+1}&=\bu_{i+1} + \sum_{j=0}^{i} \, \alpha_j \, (\bu_{i+1}-\bu_j),
\label{eq:GMRESopt}
\end{align}
%------------------------------------------------------------------
such that $\|\hat{\br}_{i+1}\|_2$ is minimized (which leads to a small least-squares problem equivalent to
(\ref{eq:minAlpha})). Note that $V_{1,i+1}$ and $V_{2,i+1}$ do not easily generalize to the nonlinear case, but the image of $V_{1,i+1}$ under $M^{-1}$, $\mathop{span}\{ \bu_{i+1}-\bu_0,  \bu_{i+1}-\bu_1,\ldots,  \bu_{i+1}-\bu_i \}$, does generalize naturally and is taken as the `generalized Krylov space' that is used to seek the approximation in the nonlinear case.

Up to this point, we have presented GMRES as a way to accelerate one-step stationary iterative method (\ref{eq:stat}). A more customary way, however, to see GMRES is in terms of preconditioning. The approach described above reduces to `non-preconditioned' GMRES when one sets $M=I$. Applying non-preconditioned GMRES to the preconditioned linear equation system $A M^{-1} (M \bu)=\bb$ also results in the expressions for preconditioned GMRES derived above. In this viewpoint, the matrix $M^{-1}$ is called the preconditioner matrix, because its role is viewed as to pre-condition the spectrum of the linear system operator such that the (non-preconditioned) GMRES method applied to 
$(A M^{-1}) \by=\bb$ becomes more effective. 
It is also customary to say that the stationary iterative process preconditions GMRES (for example, Gauss-Seidel or Jacobi can  precondition GMRES). We can summarize that the role of the stationary iterative method is to generate preconditioned residuals that build the Krylov space.

In the presentation above, all iterates $\bu_j$ for $j=0,\ldots,i$ (for instance, in the right-hand side of (\ref{eq:GMRESopt})) refer to the unaccelerated iterates generated by stationary iterative method (\ref{eq:stat}). However, the formulas remain valid when accelerated iterates are used instead; this does change the values of the coefficients $\alpha_j$, but leads to the same accelerated iterates \cite{WashioNGMRES-ETNA}. This is so because the Krylov spaces generated in the two cases are identical due to linearity, and consequently GMRES selects the same optimal improved iterate. 
%In the nonlinear case, the relevant generalization of the Krylov space is obtained as the image of $V_{3,i+1}$ under $M^{-1}$: $\mathop{span}\{ (\bu_{i+1}-\bu_0),(\bu_{i+1}-\bu_1),\ldots,  (\bu_{i+1}-\bu_i) \}$ is indeed well-defined in the nonlinear case as well, and is the space in which nonlinear GMRES seeks improved approximations.

This brings us to the point where we can compare steepest-descent preconditioned N-GMRES applied to quadratic objective function (\ref{eq:fsimple}) with SPD operator $A$, to non-preconditioned linear GMRES applied to $A \bu=\bb$.  Assume we have $w$ previous iterates $\bu_i$ and residuals $\br_i$. Stationary iterative process (\ref{eq:stat}) without preconditioner ($M=I$) would add a vector to the Krylov space which has the same direction as the vector that would be added to it by the steepest descent preconditioning process (\ref{eq:precondsimple}). This means that the accelerated iterate $\hat{\bu}_{i+1}$ produced by N-GMRES with steepest descent preconditioner applied to 
quadratic objective function (\ref{eq:fsimple}) with SPD operator $A$ is the same as the accelerated iterate $\hat{\bu}_{i+1}$ produced by linear GMRES with identity preconditioner applied to $A \bu=\bb$.
%If one further assumes that the line search process in Step III of Algorithm \ref{alg:N-GMRES} applied to $f(\bu)$ from
%(\ref{eq:fsimple}) always selects the accelerated iterate $\hat{\bu}_{i+1}$ (which can indeed be expected for the simple quadratic objective function, since $\beta=1$ is normally chosen as the initial guess for step length $\beta$), then all iterates produced by steepest-descent preconditioned N-GMRES applied to $f(\bu)$ will be the same as the iterates produced by linear GMRES applied to $A \bu=\bb$.
%This shows that, for the simple quadratic $f(\bu)$ from (\ref{eq:fsimple}) with $A$ SPD, steepest-descent preconditioned N-GMRES produces the same accelerated iterate as linear GMRES with identity preconditioner $M=I$ applied to $A \bu=\bb$. 
This motivates our proposal to use steepest descent preconditioning as the natural and most basic preconditioning process for the N-GMRES optimization algorithm applied to general nonlinear optimization problems.

Note that, in the case of linear systems, the efficiency of GMRES as an acceleration technique for stationary iterative methods can be understood in terms of how optimal polynomials can damp modes that are slow to converge \cite{WashioNGMRES-ETNA,SaadBook}. In the case of N-GMRES for nonlinear optimization, if the approximation is close to a stationary point and the nonlinear residual vector function $\bg(.)$ can be approximated well by linearization, then it can be expected that the use of the subspace $\mathop{span}\{ \bu_{i+1}-\bu_0,  \bu_{i+1}-\bu_1,\ldots,  \bu_{i+1}-\bu_i \}$ for acceleration may give efficiency similar to the linear case \cite{WashioNGMRES-ETNA}.
Note finally that the above also explains why a small step is allowed in the $sd$ preconditioner of (\ref{eq:steepestB}) (basically, in the linear case, the size of the coefficient does not matter for the Krylov space), and the linearization argument of (\ref{eq:linearize}) indicates that a small step may be beneficial.

%*****************************************************************
\subsection{Convergence Theory for N-GMRES Optimization with Steepest Descent Preconditioning}
\label{subsec:ConvTheory}
%*****************************************************************
We now formulate and prove a convergence theorem for N-GMRES Optimization Algorithm \ref{alg:N-GMRES} using steepest descent preconditioning with line search (\ref{eq:steepestA}). We assume that all line searches provide step lengths that satisfy the Wolfe conditions \cite{Nocedal}: 
%------------------------------------------------------------------
\begin{alignat}{1}
&\text{\sc sufficient decrease condition:} \nonumber \\
&\qquad f(\bu_i+\beta_i \bp_i)\le f(\bu_i)+c_1\, \beta_i \, \nabla f(\bu_i)^T  \bp_i, \label{eq:Wolfea}\\
&\text{\sc curvature condition:} \nonumber \\
&\qquad \nabla f(\bu_i+\beta_i \bp_i)^T \, \bp_i \ge c_2 \, \nabla f(\bu_i)^T  \bp_i, \label{eq:Wolfeb}
\end{alignat}
%------------------------------------------------------------------
with $0<c_1<c_2<1$.
Condition (\ref{eq:Wolfea}) ensures that large steps are taken only if they lead to a proportionally large decrease in $f$.
Condition (\ref{eq:Wolfeb}) ensures that a step is taken that is large enough to sufficiently increase the gradient of $f$ in the line search direction (make it less negative).
%sequence $\{\| \nabla f(\bu_i) \|\}$ not bounded away from 0
Global convergence (in the sense of convergence to a stationary point from any initial guess) can then be proved easily using standard approaches \cite{NocedalNCG,Nocedal}.\\
%++++++++++++++++++++++++++++++++++++++++++
%------------------------------------------------------------------
\begin{theorem}[Global convergence of N-GMRES optimization algorithm with steepest descent line search preconditioning]
Consider N-GMRES Optimization Algorithm \ref{alg:N-GMRES} with steepest descent line search preconditioning (\ref{eq:steepestA}) for Optimization Problem I, and assume that all line search solutions satisfy the Wolfe conditions, (\ref{eq:Wolfea}) and (\ref{eq:Wolfeb}). Assume that objective function $f$ is bounded below in $\mathbb{R}^n$ and that $f$ is continuously differentiable in an open set ${\cal N}$ containing the level set ${\cal L}=\{ \bu : f(\bu) \le f(\bu_0)\}$, where $\bu_0$ is the starting point of the iteration. Assume also that the gradient $\nabla f$ is Lipschitz continuous on $\cal{N}$, that is, there exists a constant $L$ such that $\|\nabla f(\bu) - \nabla f(\hat{\bu})\| \le L \|\bu-\hat{\bu}\|$ for all $\bu, \hat{\bu} \in \cal{N}$. Then the sequence of N-GMRES iterates $\{ \bu_0, \bu_1, \ldots\}$ is convergent to a fixed point of Optimization Problem I in the sense that
%------------------------------------------------------------------
\begin{align}
\lim_{i \rightarrow \infty} \| \nabla f(\bu_i) \| = 0.
\label{eq:lim}
\end{align}
%------------------------------------------------------------------
\label{thm:conv}
\end{theorem}
\begin{proof}
Consider the sequence $\{ \bv_0,\bv_1,\ldots \}$ formed by the iterates $\bu_0$, $\bar{\bu}_1$, $\bu_1$, $\bar{\bu}_2$, $\bu_2$, $\ldots$ of Algorithm I, but with $\bar{\bu}_i$ removed if $\hat{\bu}_{i}-\bar{\bu}_i$ is not a descent direction in Step III of the algorithm. Then all iterates $\bv_i$ are of the form $\bv_i=\bv_{i-1} + \beta_{i-1} \bp_{i-1}$, with $\bp_{i-1}$ a descent direction and $\beta_{i-1}$ such that the Wolfe conditions are satisfied. According to Theorem 3.2 of \cite{Nocedal} (p. 38, Zoutendijk's Theorem), we have that
%------------------------------------------------------------------
\begin{align}
\sum_{i=0}^{\infty} \cos^2 \theta_i \, \| \nabla f(\bv_i) \|^2 < \infty,
\label{eq:zout}
\end{align}
%------------------------------------------------------------------
with
%------------------------------------------------------------------
\begin{align}
\cos \theta_i=\frac{-\nabla f(\bv_i)^T  \bp_i}{\| \nabla f(\bv_i)\| \, \|\bp_i\|},
\label{eq:theta}
\end{align}
%------------------------------------------------------------------
which implies that
%------------------------------------------------------------------
\begin{align}
\lim_{i \rightarrow \infty}  \cos^2 \theta_i \, \|  \nabla f(\bv_i) \|^2 = 0.
\label{eq:lim2}
\end{align}
%------------------------------------------------------------------
Consider the subsequence $\{\|\nabla f(\bu_i)\|\}$ of $\{\|\nabla f(\bv_i)\|\}$.
Since all the $\bu_i$ are followed by a steepest descent step in the algorithm, the $\theta_i$ corresponding to all the elements of $\{\|\nabla f(\bu_i)\|\}$ satisfy $\cos \theta_i=1$. Therefore, it follows from (\ref{eq:lim2}) that $\lim_{i \rightarrow \infty} \| \nabla f(\bu_i) \| = 0$, which concludes the proof.
\end{proof}
%------------------------------------------------------------------
%++++++++++++++++++++++++++++++++++++++++++

Note that the notion of convergence (\ref{eq:lim}) we prove in Theorem \ref{thm:conv} for N-GMRES optimization with steepest descent line search preconditioning is stronger than the type of convergence that can be proved for some N-CG methods \cite{NocedalNCG,Nocedal}, namely, 
%------------------------------------------------------------------
\begin{align}
\lim_{i \rightarrow \infty} \inf \| \nabla f(\bu_i) \| = 0.
\label{eq:liminf}
\end{align}
%------------------------------------------------------------------
Also, it appears that, in the proof of Theorem \ref{thm:conv}, we cannot guarantee that sequence $\{  \| \nabla f(\bar{\bu}_i)\| \}$ converges to 0. We know that sequence $\{  f(\bv_i) \}$ converges to a value $f^*$ since it is nonincreasing and bounded below, but it appears that the properties of the line searches do not guarantee that the sequence $\{  \| \nabla f(\bv_i)\| \}$ converges to 0. They do guarantee that the subsequence $\{  \| \nabla f(\bu_i)\| \}$ converges to 0, but it cannot be ruled out that, as the $f(\bu_i)$ approach $f^*$ and the $\| \nabla f(\bu_i)\|$ approach 0, large steps with very small decrease in $f$ may still be made from each $\bu_i$ to the next $\bar{\bu}_{i+1}$ (large steps with small decrease are allowed in this case since the $\bu_i$ approach a stationary point), while, at the same time, large steps with very small decrease in $f$ may be made from the $\bar{\bu}_{i+1}$ to the next $\bu_{i+1}$ (large steps with small decrease are allowed in this case if the search direction $\bp$ from $\bar{\bu}_{i+1}$ is such that $\nabla f(\bar{\bu}_{i+1})^T  \bp$ is very close to 0). These large steps may in principle preclude $\{  \| \nabla f(\bar{\bu}_i)\| \}$ from converging to 0 (but we do not observe such pathological cases in our numerical tests).
Nevertheless, we are able to prove the strong convergence result (\ref{eq:lim}) for the iterates $\bu_i$ of N-GMRES optimization with steepest descent line search preconditioning: sequence $\{  \| \nabla f(\bu_i)\| \}$ converges to 0. 
%%%%%%%%%%%%%%%%%%%%%%%%%%%%%%%%%%%%%%%%%%%%%%%%%%%%%%%%%%
%%%%%%%%%%%%%%%%%%%%%%%%%%%%%%%%%%%%%%%%%%%%%%%%%%%%%%%%%%
\section{Numerical Results}
\label{sec:numerics}
%%%%%%%%%%%%%%%%%%%%%%%%%%%%%%%%%%%%%%%%%%%%%%%%%%%%%%%%%%
%%%%%%%%%%%%%%%%%%%%%%%%%%%%%%%%%%%%%%%%%%%%%%%%%%%%%%%%%%
We now present extensive numerical results for the N-GMRES optimization algorithm with steepest descent preconditioners (\ref{eq:steepestA}) and (\ref{eq:steepestB}), compared with stand-alone steepest descent optimization, N-CG and L-BFGS.

In all tests, we utilize the Mor\'{e}-Thuente line search method \cite{MoreThuente} and the N-CG and L-BFGS optimization methods as implemented in the Poblano toolbox for Matlab \cite{POBLANO}. For all experiments, the Mor\'{e}-Thuente line search parameters used were as follows: function value tolerance $c_1=10^{-4}$ for (\ref{eq:Wolfea}), gradient norm tolerance $c_2=10^{-2}$ for (\ref{eq:Wolfeb}), starting search step length $\beta=1$, and a maximum of 20 $f/g$ evaluations are used. These values were also used  for the N-CG and L-BFGS comparison runs.  We use the N-CG variant with Polak-Ribi\`{e}re update formula, and the two-loop recursion version of L-BFGS \cite{Nocedal}.
%As suggested in \cite{WashioNGMRES-ETNA}, the parameter $\delta$ in the term $\delta \, I$ added to normal equation matrix $\bP^T \, \bP$ in (\ref{eq:normal}), was chosen as $\epsilon$ times the size of the largest diagonal element of $\bP^T \, \bP$, with $\epsilon=10^{-12}$. 
We normally choose the N-GMRES window size $w$ equal to 20, which is confirmed to be a good choice in numerical tests described below. The L-BFGS window size is chosen equal to 5 (we found that larger window sizes tend to harm L-BFGS performance for the tests we considered).
All initial guesses are determined uniformly randomly with components in the interval $[0,1]$, and when we compare different methods they are given the same random initial guess. All numerical tests were run on a laptop with a dual-core 2.53 GHz Intel Core i5 processor and 4GB of 1067 MHz DDR3 memory. Matlab version 7.11.0.584 (R2010b) 64-bit (maci64) was used for all tests.

%*****************************************************************
\subsection{Test Problem Description}
\label{subsec:tests}
%*****************************************************************
We first describe the seven test problems we consider. In what follows, all vectors are chosen in $\mathbb{R}^n$, and all matrices in $\mathbb{R}^{n \times n}$.

\noindent
{\sc Problem A. (Quadratic objective function with spd diagonal matrix.)} 
%------------------------------------------------------------------
\begin{align}
f(\bu)=\frac{1}{2} \, (\bu-\bu^*)^T D \, (\bu-\bu^*)+1,\\
 \text{with} \ D=\text{diag}(1,2,\ldots,n).\nonumber
\label{eq:fA}
\end{align}
%------------------------------------------------------------------
This problem has a unique minimizer $\bu^*$ in which $f^*=f(\bu^*)=1$. We choose $\bu^*=(1,\ldots,1)$.  Note that $\bg(\bu)=D (\bu-\bu^*),$ and the condition number of $D$ is given by $\kappa=n$. It is well-known that for problems of this type large condition numbers tend to lead to slow convergence of the steepest descent method due to a zig-zag effect. Problem A can be used to show how methods like N-CG and N-GMRES improve over steepest descent and mitigate this zig-zag effect.\\

\noindent
{\sc Problem B. (Problem A with paraboloid coordinate transformation.)} 
%------------------------------------------------------------------
\begin{align}
f(\bu)=\frac{1}{2} \, \by(\bu-\bu^*)^T D \, \by(\bu-\bu^*)+1,\\
 \text{with} \ D=\text{diag}(1,2,\ldots,n) \ \text{and} \ \by(\bx) \ \text{given by} \nonumber\\
y_1(\bx)=x_1 \ \text{and} \  y_i(\bx)=x_i-10 \, x_1^2 \ (i=2,\ldots,n). \nonumber
\label{eq:fB}
\end{align}
%------------------------------------------------------------------
This modification of Problem A still has a unique minimizer $\bu^*$ in which $f^*=f(\bu^*)=1$. We choose $\bu^*=(1,\ldots,1)$.  The gradient of $f(\bu)$ is given by $\bg(\bu)=D \, \by(\bu-\bu^*)-20 \, (u_1-u_1^*) \, (\sum_{j=2}^{n} (D \, \by(\bu-\bu^*))_j) \, [1,0,\ldots,0]^T$. This modification of Problem A increases nonlinearity (the objective function is now quartic in $\bu$) and changes the level surfaces from ellipsoids into parabolically skewed ellipsoids. As such, the problem is more difficult for nonlinear optimization methods. For $n=2$, the level curves are modified from elliptic to `banana-shaped'. In fact, the objective function of Problem B is a multi-dimensional generalization of Rosenbrock's `banana' function.\\

\noindent
{\sc Problem C. (Problem B with a random non-diagonal matrix with condition number $\kappa=n$.)} 
%------------------------------------------------------------------
\begin{align}
f(\bu)=\frac{1}{2} \, \by(\bu-\bu^*)^T T \, \by(\bu-\bu^*)+1,\\
 \text{with} \ T= Q \, \text{diag}(1,2,\ldots,n) \, Q^T, \ \text{where $Q$ is a}  \nonumber \\
  \text{random orthogonal matrix and} \ \by(\bx) \ \text{is given by} \nonumber\\
y_1(\bx)=x_1 \ \text{and} \  y_i(\bx)=x_i-10 \, x_1^2 \ (i=2,\ldots,n). \nonumber
\label{eq:fC}
\end{align}
%------------------------------------------------------------------
This modification of Problem B still has a unique minimizer $\bu^*$ in which $f^*=f(\bu^*)=1$. We choose $\bu^*=(1,\ldots,1)$.  The gradient of $f(\bu)$ is given by $\bg(\bu)=T \, \by(\bu-\bu^*)-20 \, (u_1-u_1^*) \, (\sum_{j=2}^{n} (T \, \by(\bu-\bu^*))_j) \, [1,0,\ldots,0]^T$. The random matrix $Q$ is the $Q$ factor obtained from a QR-factorization of a random matrix with elements uniformly drawn from the interval $[0,1]$. This modification of Problem B introduces nonlinear `mixing' of the coordinates (cross-terms) and further increases the difficulty of the problem.\\

%---------------------------------------------------------------------------------------------------------------------------------
\begin{figure}[!htbp]
  \centering
  \scalebox{0.35}{
  \includegraphics{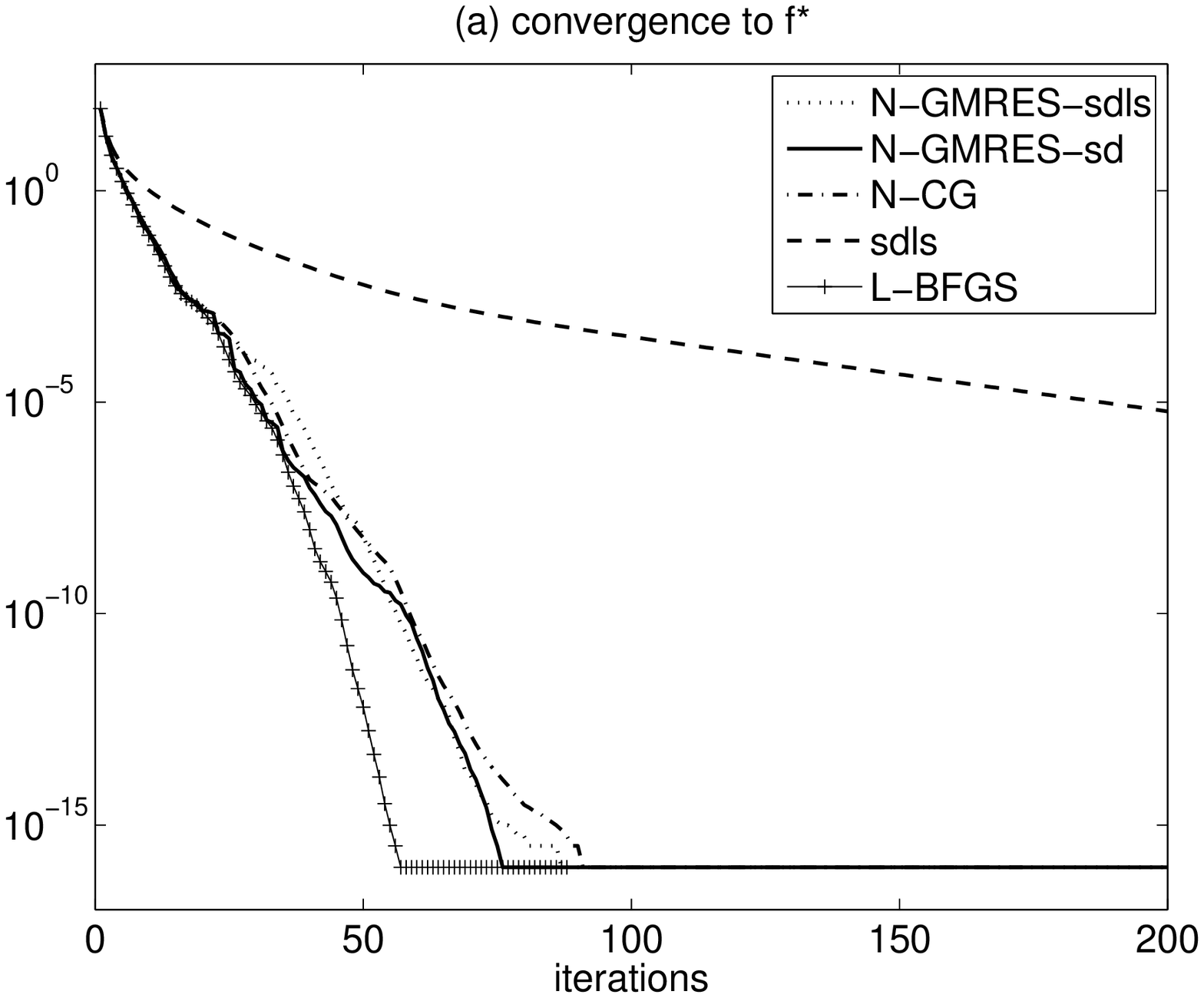}
  \includegraphics{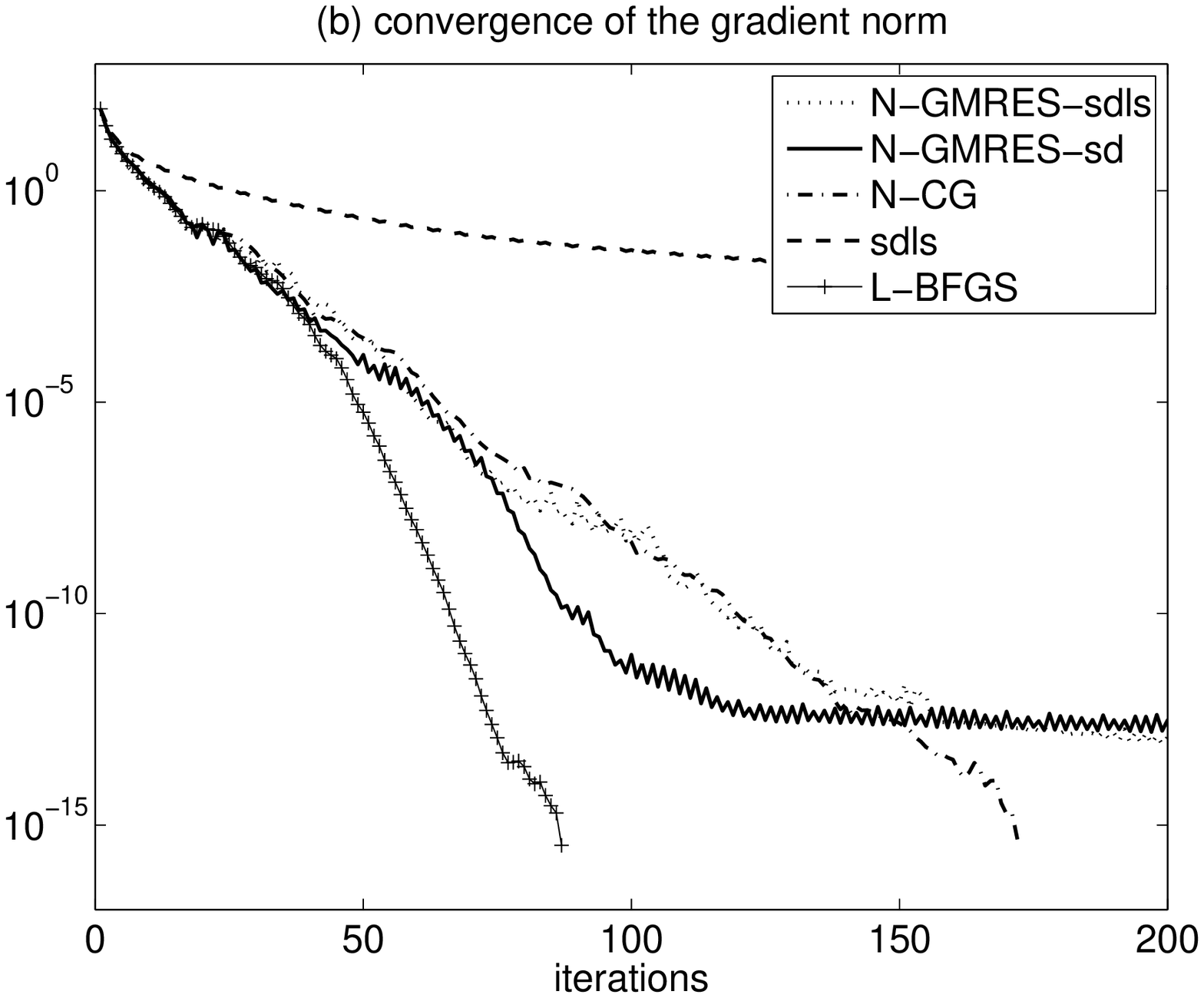}
  }
  \scalebox{0.35}{
  \includegraphics{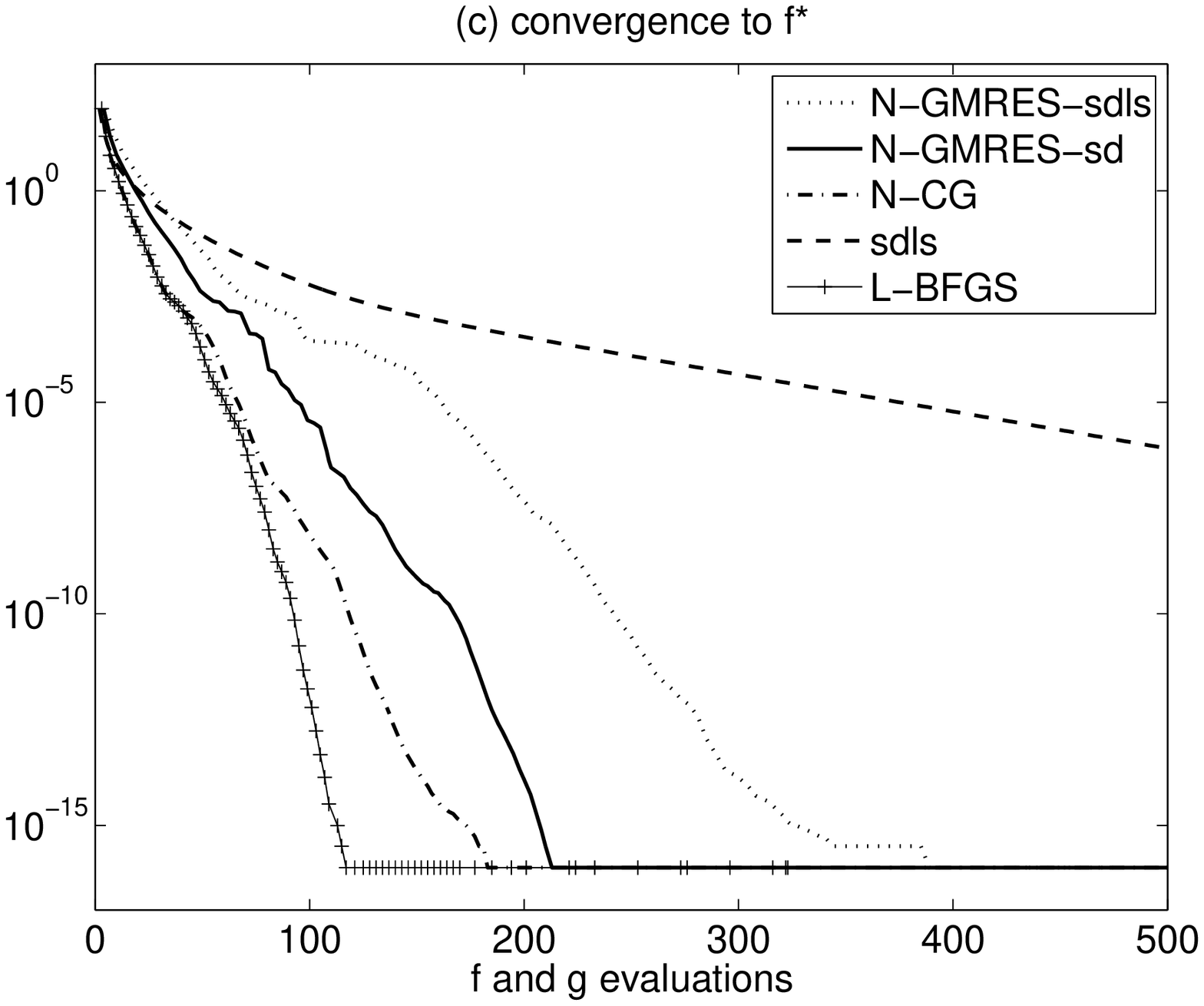}
  \includegraphics{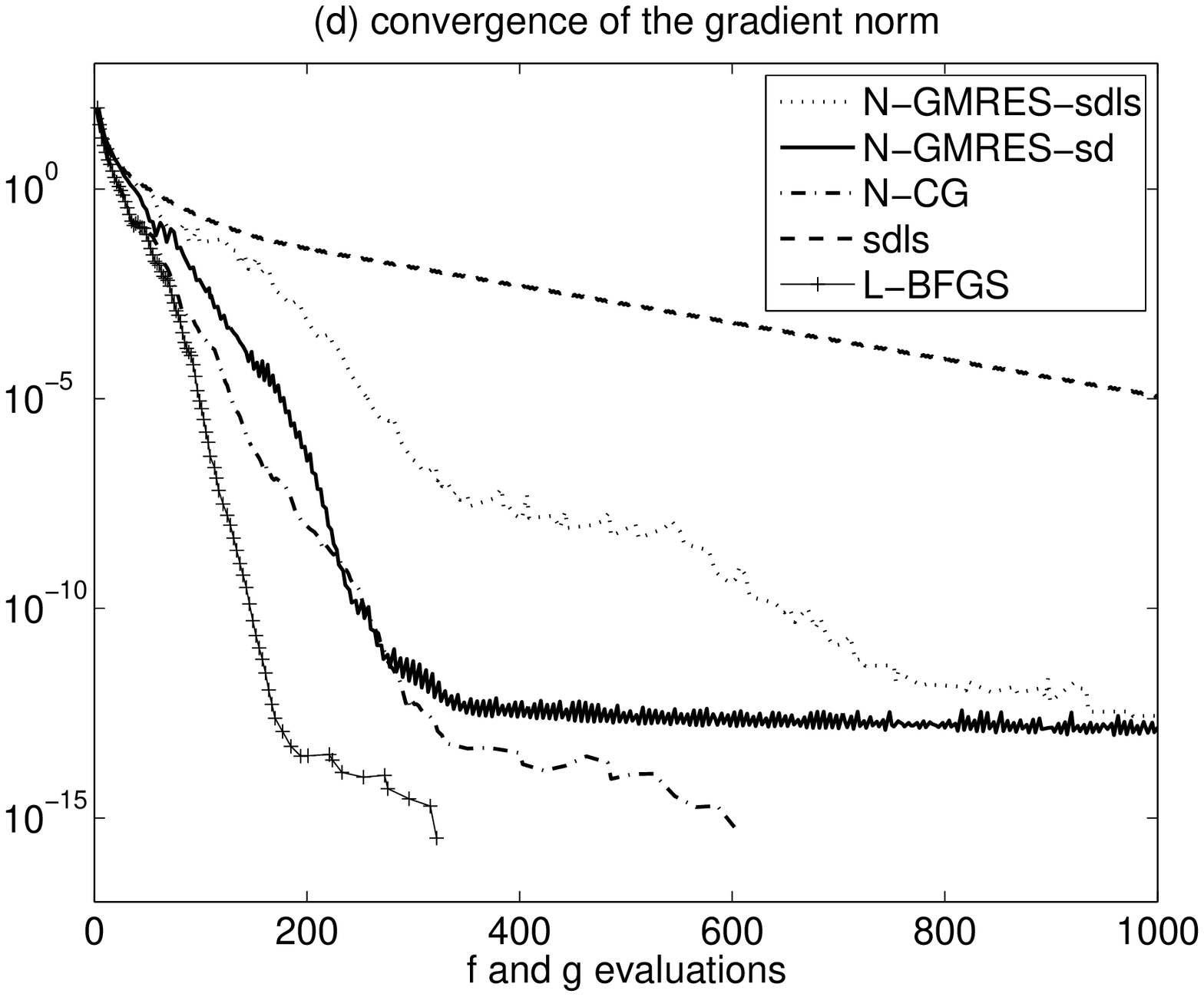}
  }
   \caption{Problem A ($n=100$). Convergence histories of the 10-logarithms of $|f(\bu_i)-f^*|$ and $\|\bg(\bu_i)\|$ as a function
   of iterations and $f/g$ evaluations. N-GMRES-sdls is the N-GMRES optimization algorithm using steepest descent preconditioning with line search, N-GMRES-sd is the N-GMRES optimization algorithm using steepest descent preconditioning with predefined step, N-CG is the Polak-Ribi\`{e}re nonlinear conjugate gradient method, L-BFGS is the limited-memory Broyden-Fletcher-Goldfarb-Shanno method, and sdls is the stand-alone steepest descent method with line search.}
   \label{fig:A}
\end{figure}    
%---------------------------------------------------------------------------------------------------------------------------------
%---------------------------------------------------------------------------------------------------------------------------------
\begin{figure}[!htbp]
  \centering
  \scalebox{0.35}{
  \includegraphics{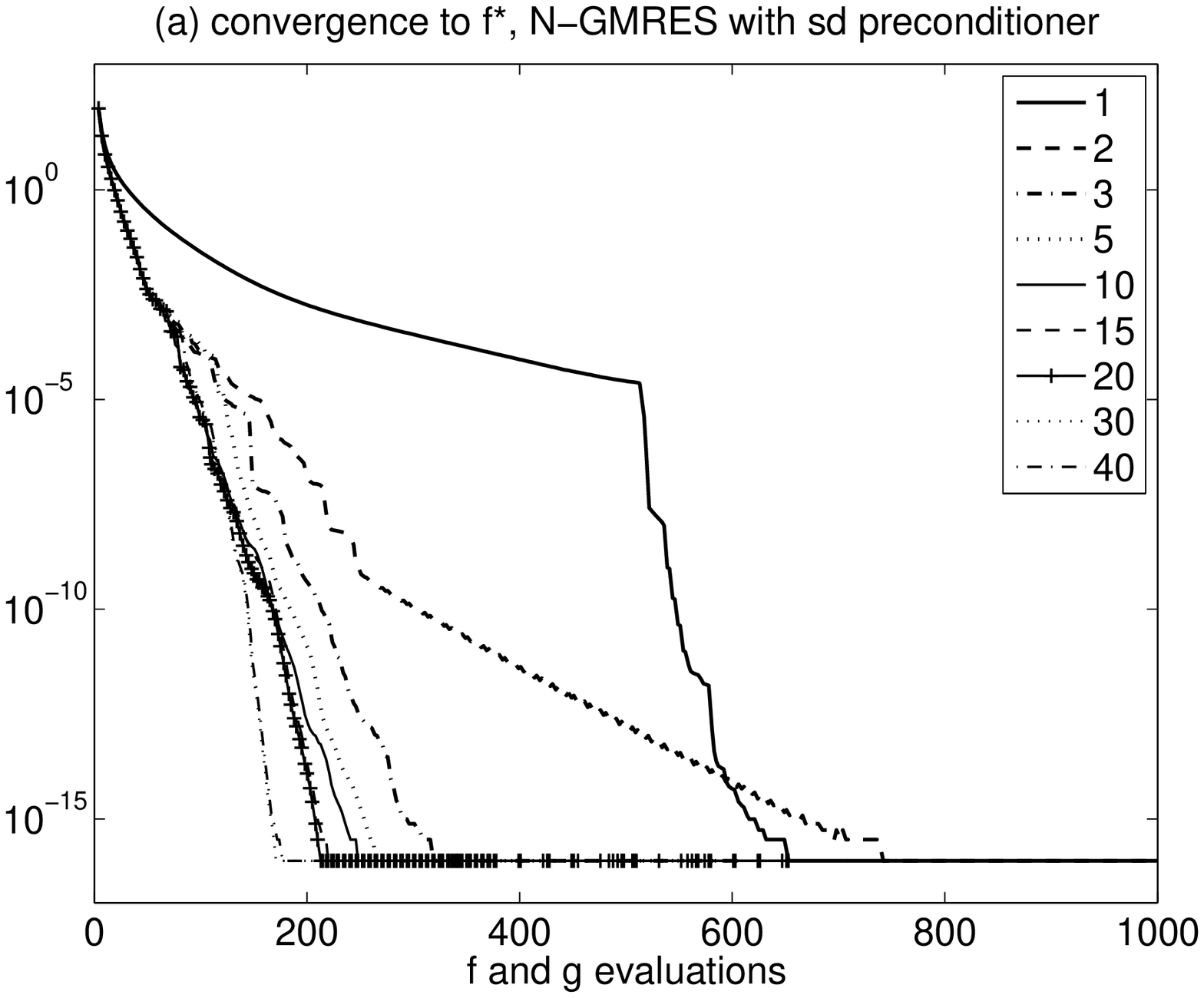}
  \includegraphics{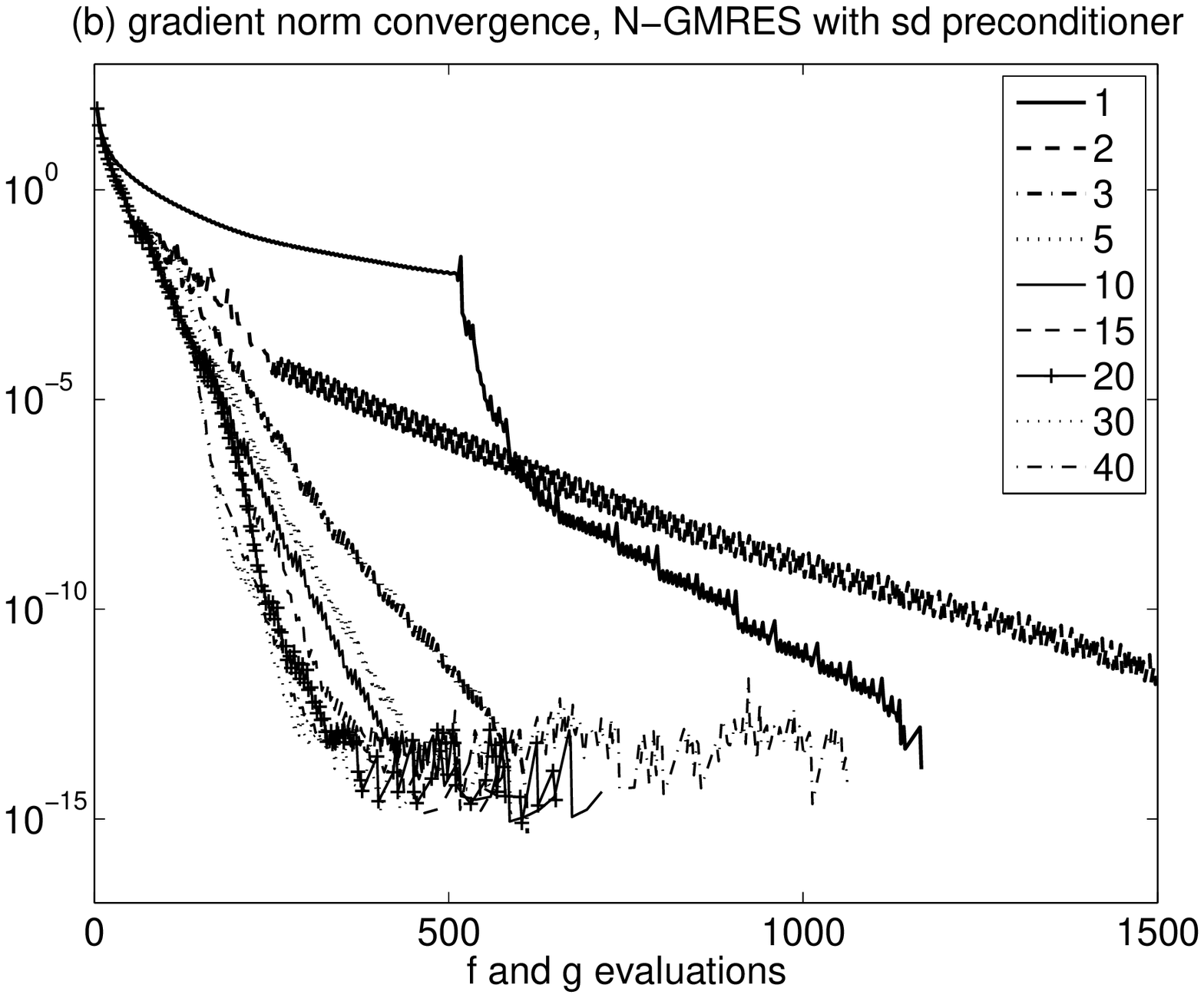}
  }
  \scalebox{0.35}{
  \includegraphics{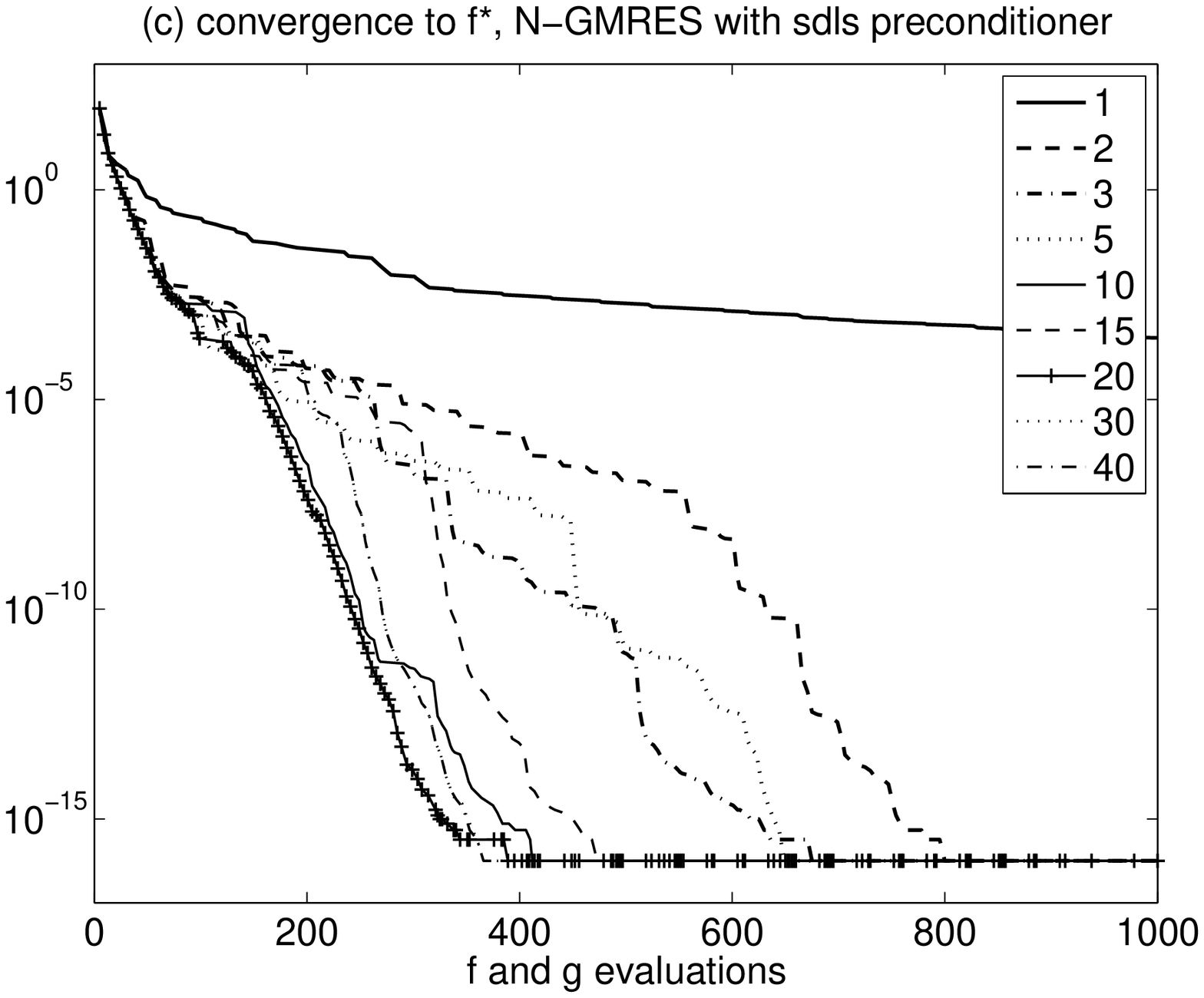}
  \includegraphics{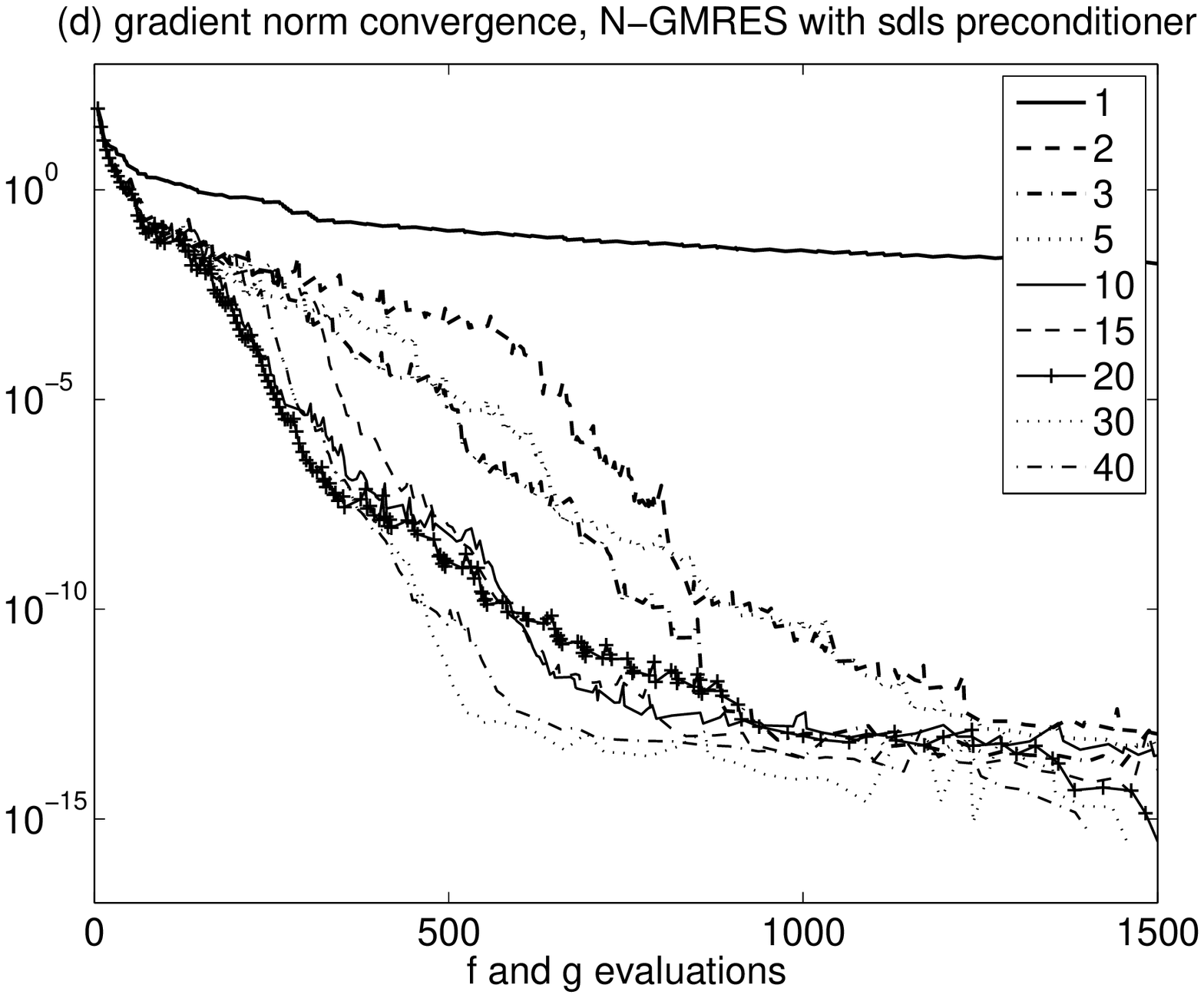}
  }
   \caption{Problem A ($n=100$). Effect of varying window size $w$ on $|f(\bu_i)-f^*|$ and $\|\bg(\bu_i)\|$ convergence for N-GMRES-sdls and N-GMRES-sd optimization as a function of $f/g$ evaluations. Window size $w=20$ emerges as a suitable choice, leading to rapid convergence. These results give some general indication that, if sufficient memory is available, $w=20$ may be a good choice. However, if memory is scarce, $w=3$ already provides good results, especially for N-GMRES-sd.}
\label{fig:w}
\end{figure}    
%---------------------------------------------------------------------------------------------------------------------------------

\noindent
{\sc Problem D. (Extended Rosenbrock function, problem (21) from \cite{MoreTest}.)} 
%------------------------------------------------------------------
\begin{alignat}{1}
f(\bu)&=\frac{1}{2} \, \sum_{j=1}^n \, t_j^2(\bu), \ \text{with $n$ even and}\nonumber\\
t_j&=10 \, (u_{j+1}-u_j^2) \qquad \text{($j$ odd),} \nonumber \\
t_j&=1-u_{j-1} \qquad \text{($j$ even).} \nonumber
%\label{eq:fD}
\end{alignat}
%------------------------------------------------------------------
Note that $\bg(\bu)$ can easily be computed using $g_k(\bu)=\sum_{j=1}^n \, t_j \, \partial t_j / \partial u_k$ ($k=1,\ldots,n$).\\

\noindent
{\sc Problem E. (Brown almost-linear function, problem (27) from \cite{MoreTest}.)} 
%------------------------------------------------------------------
\begin{alignat}{1}
f(\bu)&=\frac{1}{2} \, \sum_{j=1}^n \, t_j^2(\bu), \ \text{with}\nonumber\\
t_j&=u_j + (\sum_{i=1}^{n} u_i)-(n+1) \qquad \text{($j<n$),} \nonumber \\
t_n&=(\prod_{i=1}^{n} u_i)-1. \nonumber
%\label{eq:fE}
\end{alignat}
%------------------------------------------------------------------

\noindent
{\sc Problem F. (Trigonometric function, problem (26) from \cite{MoreTest}.)} 
%------------------------------------------------------------------
\begin{alignat}{1}
f(\bu)&=\frac{1}{2} \, \sum_{j=1}^n \, t_j^2(\bu), \ \text{with}\nonumber\\
t_j&=n-(\sum_{i=1}^{n} \, \cos u_i) - j\, (1-\cos u_j )-\sin u_j. \nonumber
%\label{eq:fF}
\end{alignat}
%------------------------------------------------------------------

\noindent
{\sc Problem G. (Penalty function I, problem (23) from \cite{MoreTest}.)} 
%------------------------------------------------------------------
\begin{alignat}{1}
f(\bu)&=\frac{1}{2} \, ((\sum_{j=1}^n \, t_j^2(\bu))+ t_{n+1}^2(\bu)), \ \text{with}\nonumber\\
t_j&=\sqrt{10^{-5}} \, (u_j-1) \qquad \text{($j=1,\ldots,n$),} \nonumber \\
t_{n+1}&=(\sum_{i=1}^{n} \, u_i^2)-0.25. \nonumber
%\label{eq:fG}
\end{alignat}
%------------------------------------------------------------------

%*****************************************************************
\subsection{Numerical Results for Problems A--C}
\label{subsec:AC}
%*****************************************************************
%---------------------------------------------------------------------------------------------------------------------------------
\begin{figure}[!htbp]
  \centering
  \scalebox{0.35}{
  \includegraphics{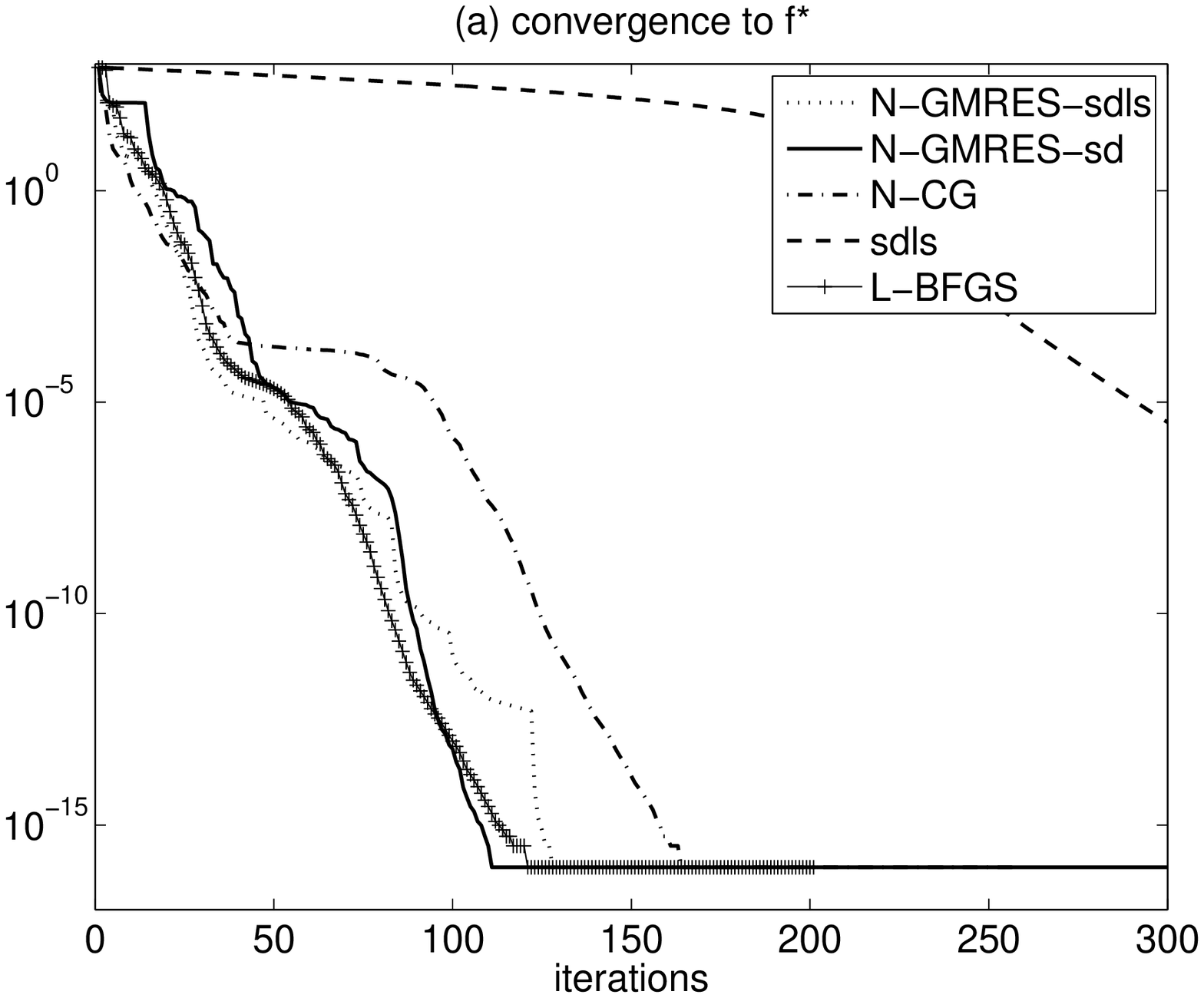}
  \includegraphics{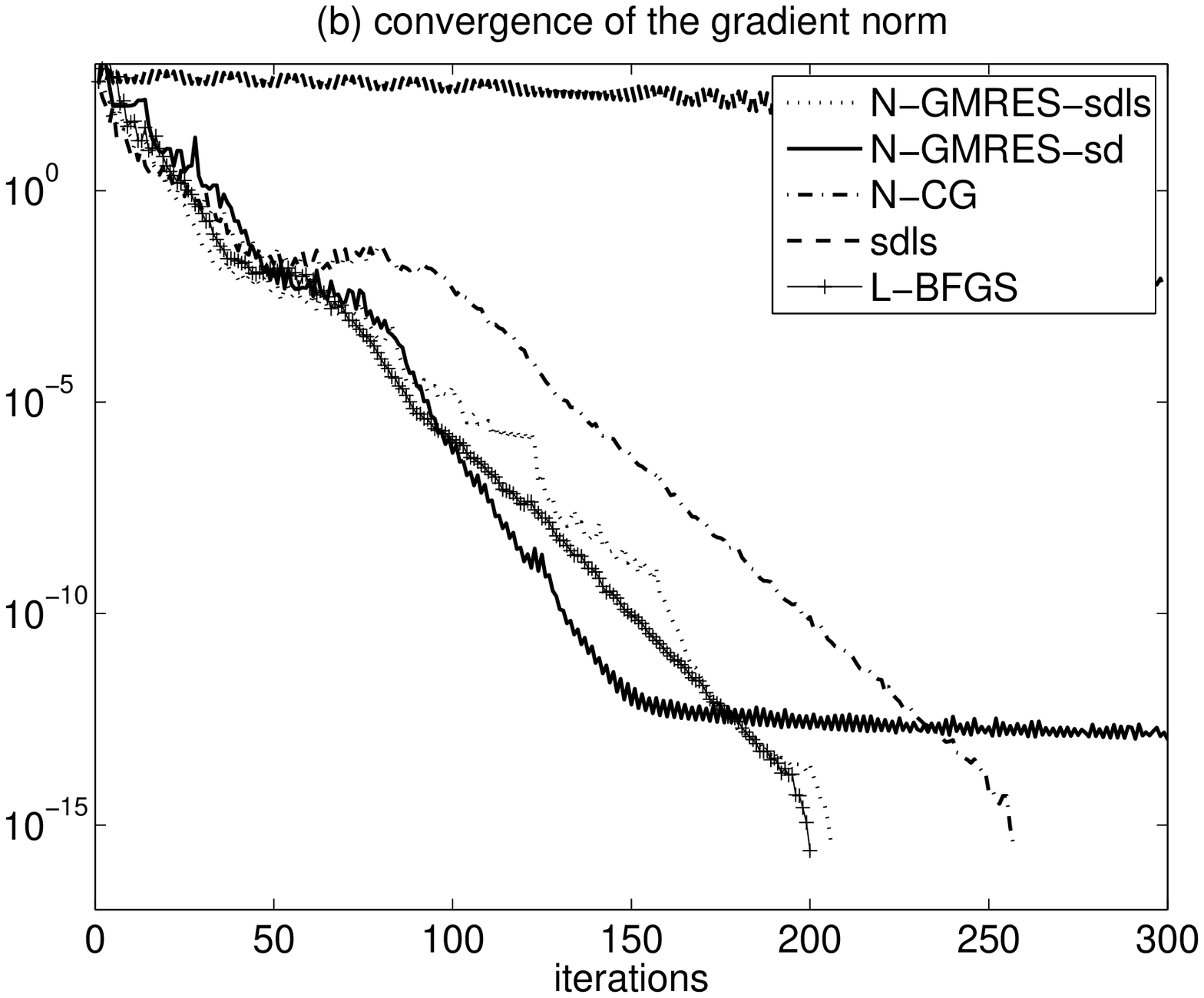}
  }
  \scalebox{0.35}{
  \includegraphics{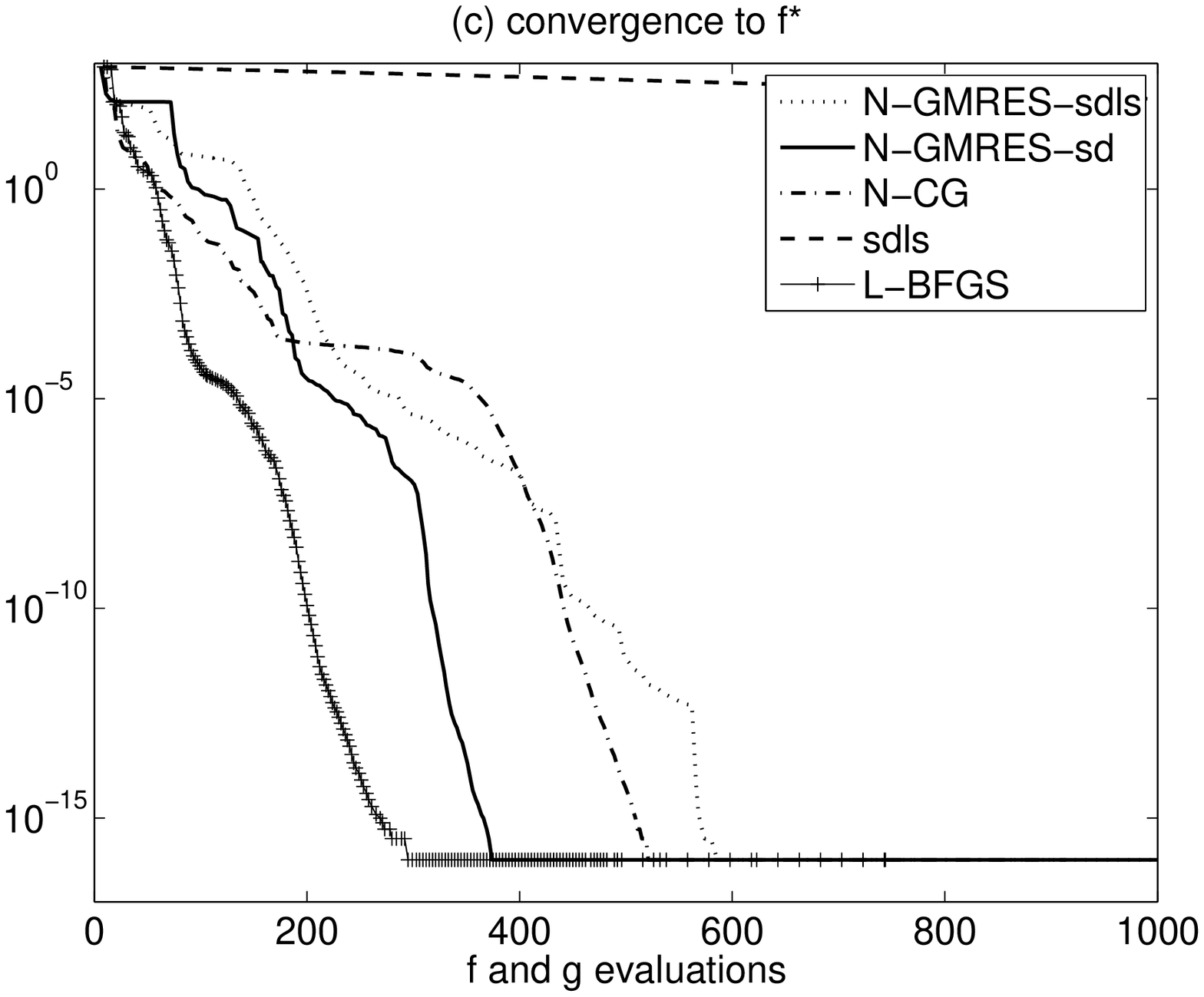}
  \includegraphics{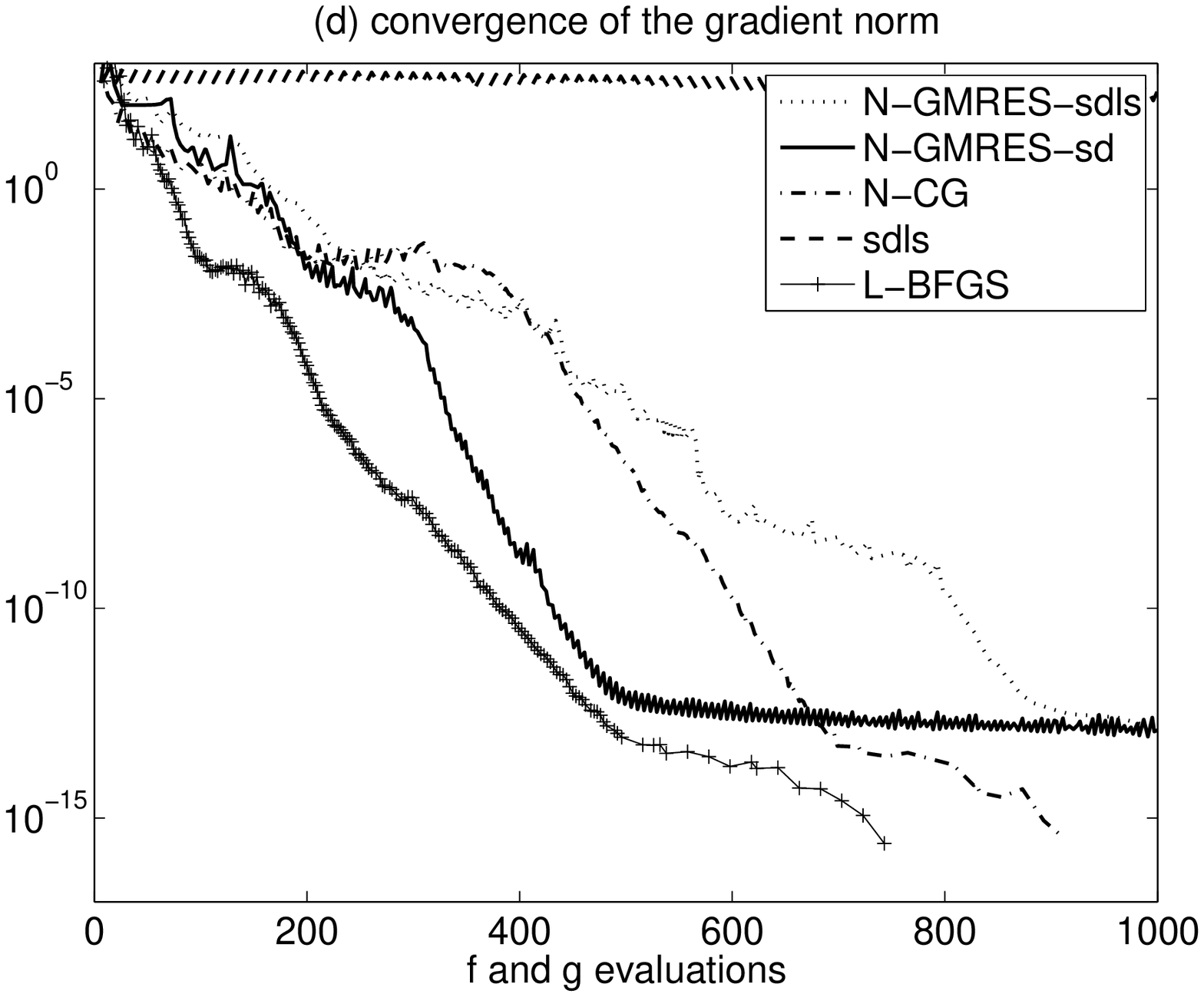}
  }
   \caption{Problem B ($n=100$). Convergence comparison.}
   \label{fig:B}
\end{figure}    
%---------------------------------------------------------------------------------------------------------------------------------
%---------------------------------------------------------------------------------------------------------------------------------
\begin{figure}[!htbp]
  \centering
  \scalebox{0.35}{
  \includegraphics{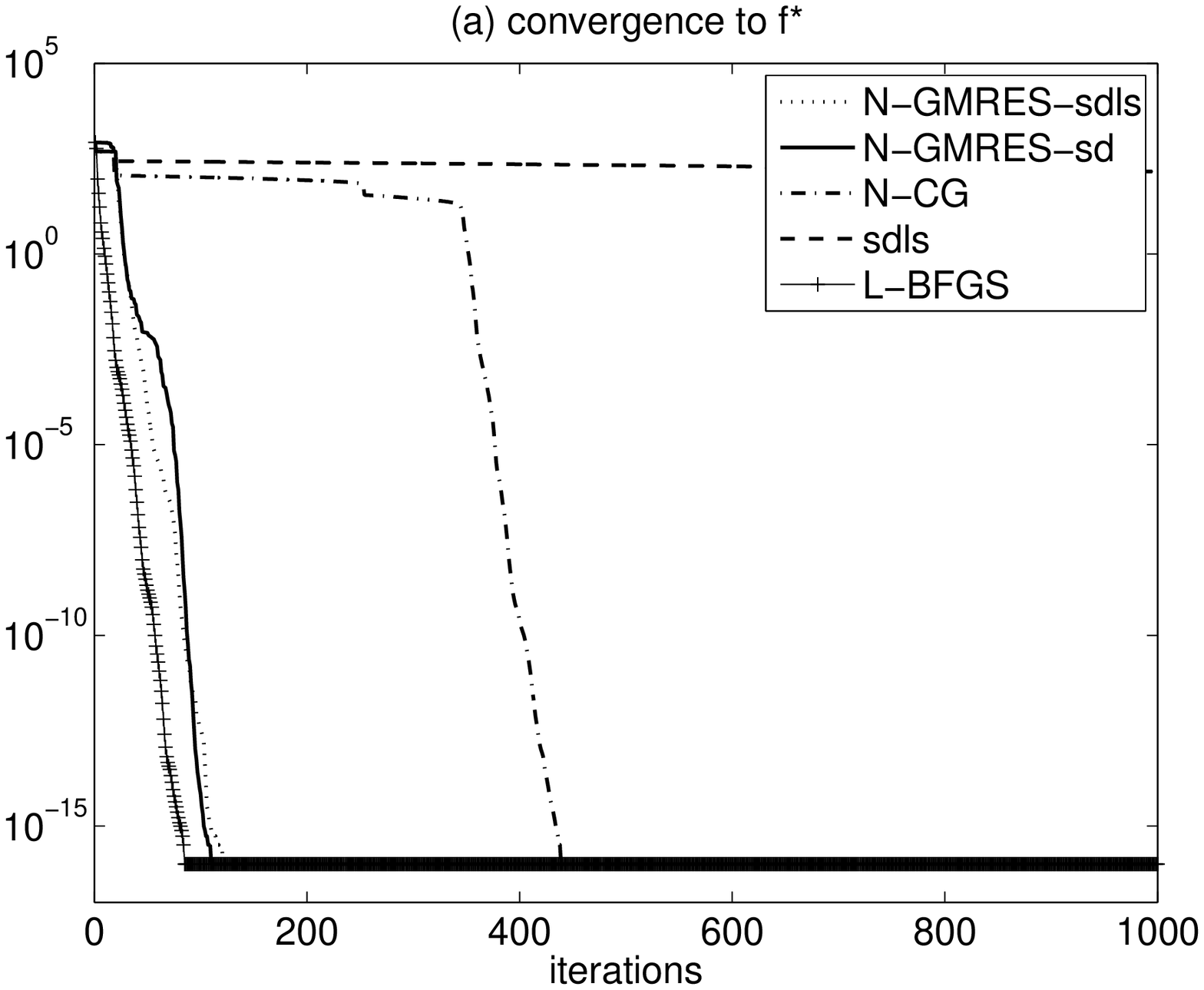}
  \includegraphics{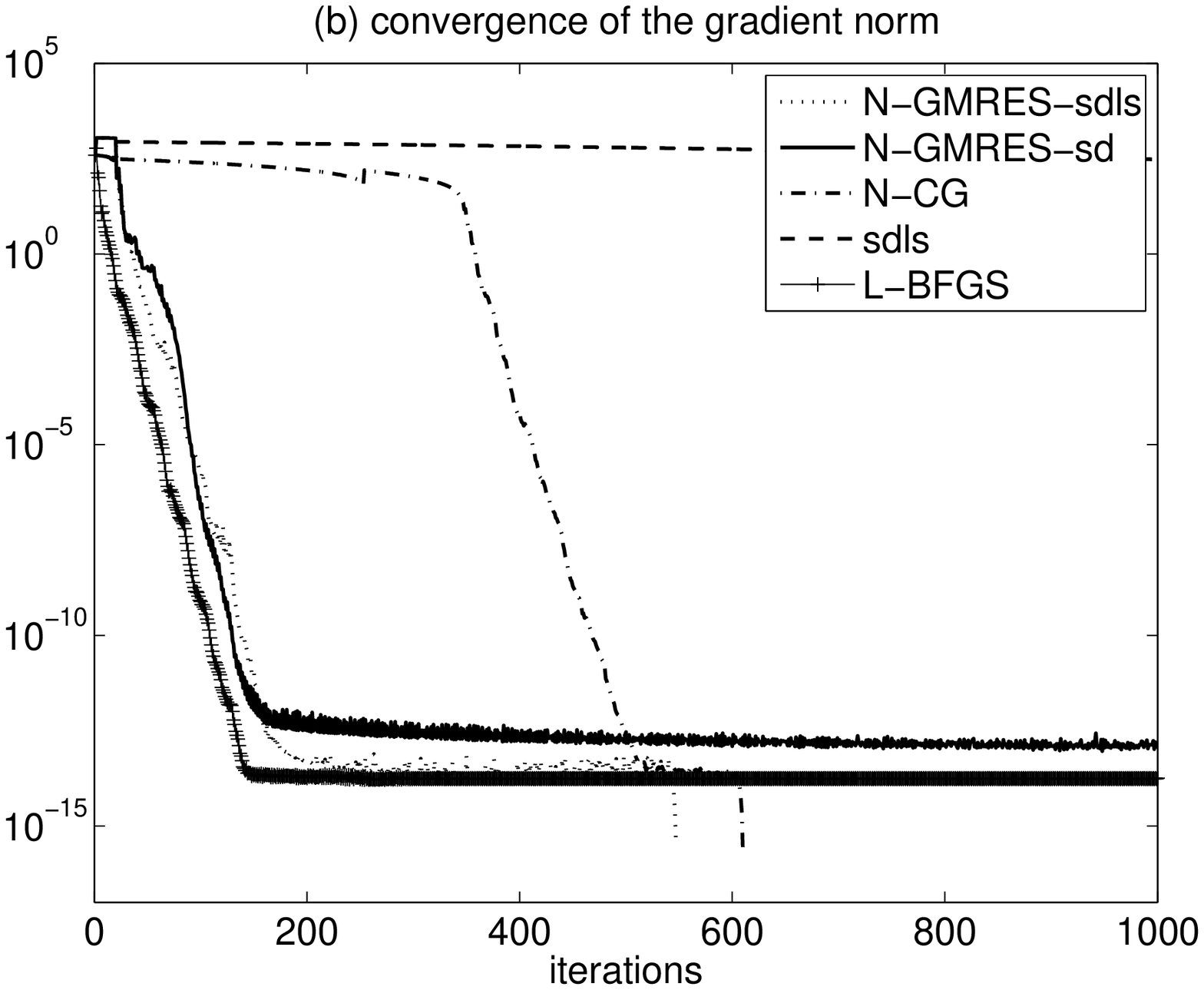}
  }
  \scalebox{0.35}{
  \includegraphics{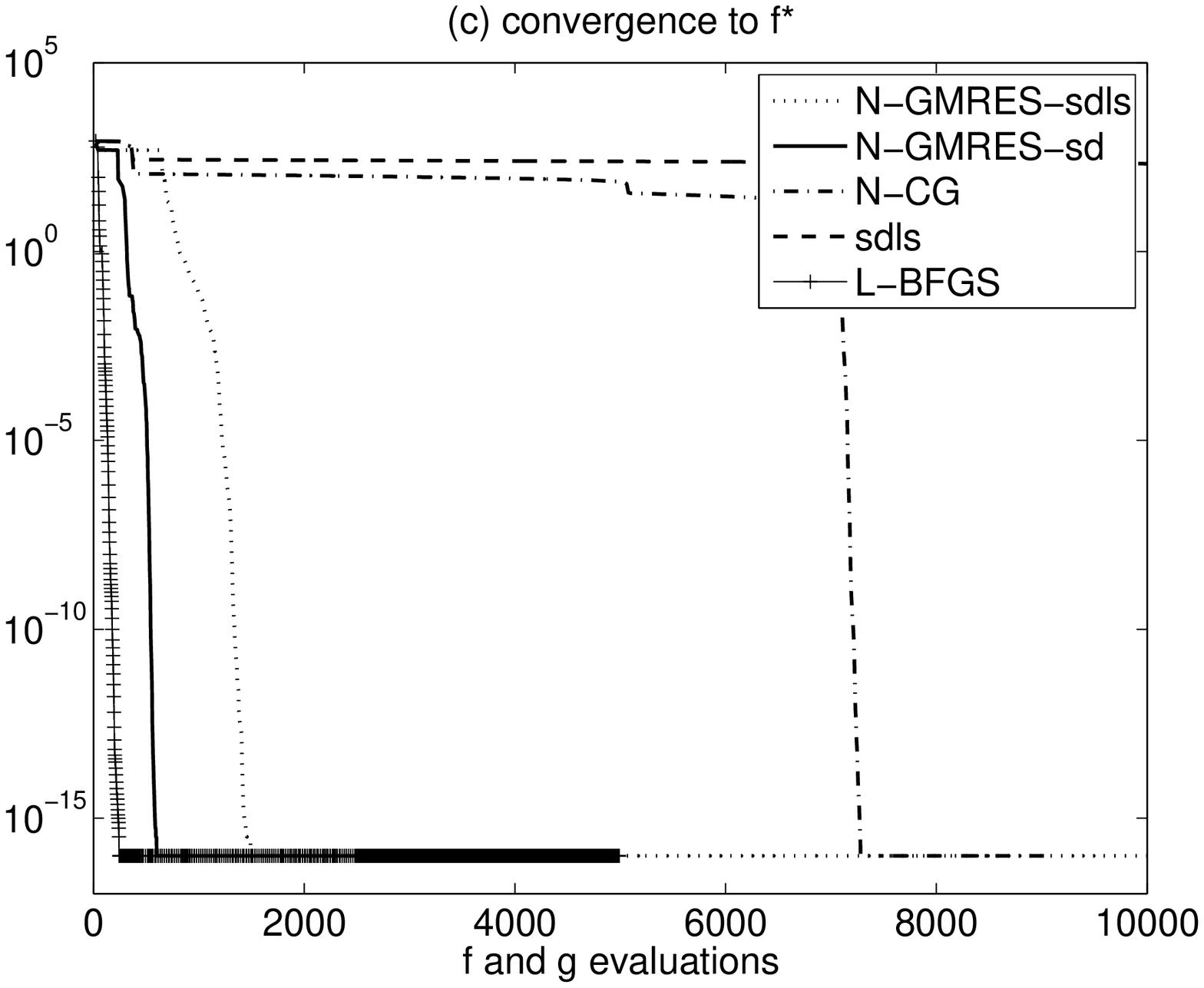}
  \includegraphics{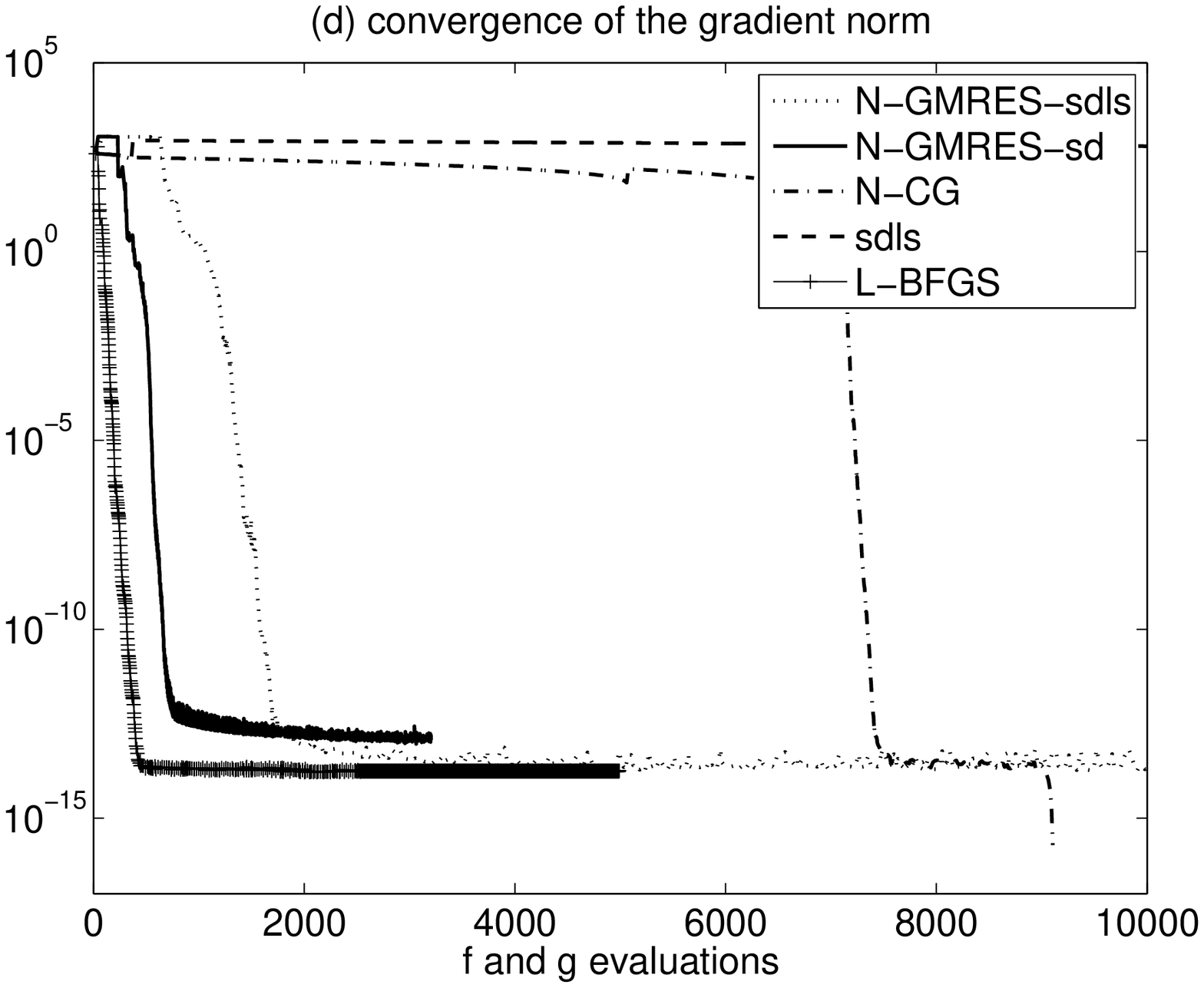}
  }
   \caption{Problem C ($n=100$). Convergence comparison.}
   \label{fig:C}
\end{figure}    
%---------------------------------------------------------------------------------------------------------------------------------

We first present some convergence plots for instances of Problems A--C.
Fig.\ \ref{fig:A} shows results for an instance of Problem A.
We see that stand-alone steepest descent with line search (sdls) converges slowly, which is expected because the condition number of matrix $D$ is $\kappa=100$. Both N-GMRES optimization using steepest descent preconditioning with line search (\ref{eq:steepestA}) (N-GMRES-sdls) and N-GMRES optimization using steepest descent preconditioning with predefined step (\ref{eq:steepestB}) (N-GMRES-sd) are significantly faster than stand-alone sdls, in terms of iterations and $f/g$ evaluations, confirming that the N-GMRES acceleration mechanism is effective, and steepest descent is an effective preconditioner for it. As could be expected, the preconditioning line searches of N-GMRES-sdls add significantly to its $f/g$ evaluation cost, and N-GMRES-sd is more effective. N-GMRES accelerates steepest descent up to a point where performance becomes competitive with N-CG and L-BFGS.
It is important to note that convergence profiles like the ones presented in Fig.\ \ref{fig:A} tend to show significant variation depending on the random initial guess. The instances presented are arbitrary and not hand-picked with a special purpose in mind (they simply correspond to seed 0 in our matlab code) and we show them because they do provide interesting illustrations and show patterns that we have verified to be quite general over many random instances. However, they cannot reliably be used to conclude on detailed relative performance of various methods. For this purpose, we provide tables below that compare performance averaged over a set of random trials.

Fig.\ \ref{fig:w} shows the effect of varying the window size $w$ on $|f(\bu_i)-f^*|$ and $\|\bg(\bu_i)\|$ convergence for N-GMRES-sdls and N-GMRES-sd optimization as a function of $f/g$ evaluations, for an instance of Problem A. Window size $w=20$ emerges as a suitable choice if sufficient memory is available, leading to rapid convergence. 
However, window sizes as small as $w=3$ already provide good results, especially for N-GMRES-sd.
This indicates that satisfactory results can be obtained with small windows, which may be useful if memory is scarce.
We use window size $w=20$ for all numerical results in this paper. 

Fig.\ \ref{fig:B} shows results for an instance of Problem B, which is a modification of Problem A introducing more nonlinearity, and Fig.\ \ref{fig:C} shows results for the even more difficult Problem C, with random nonlinear mixing of the coordinate directions. Both figures show that stand-alone sdls is very slow, and confirm that N-GMRES-sdls and N-GMRES-sd significantly speed up steepest descent. For Problem B, N-GMRES-sdls, N-GMRES-sd, N-CG and L-BFGS perform similarly, but for the more difficult Problem C N-GMRES-sdls, N-GMRES-sd and L-BFGS perform much better than N-CG.

%------------------------------------------------------------
\begin{table}[h!]
    \begin{center}
        \begin{tabular}{|l|c|c|c|c|}
\hline
problem  & N-GMRES-sdls & N-GMRES-sd & N-CG & L-BFGS \\
 \hline
%A $n$=100 & 242 & 111 & 84 \\
%A $n$=200 & 406 & 171 & 127 \\
%B $n$=100 & 1200 & 395 & 198 \\
%B $n$=200 & 1338 & 752 & 606 \\
%C $n$=100 & 926(1) & 443 & 13156(7) \\
%C $n$=200 & 1447 & 461 & 26861(9) \\
A $n$=100 & 242 & 111 & 84 & 73 \\
A $n$=200 & 406 & 171 & 127 & 104 \\
B $n$=100 & 1200 & 395 & 198 & 170 \\
B $n$=200 & 1338 & 752 & 606 & 321 \\
C $n$=100 & 926(1) & 443 & 13156(7) & 151 \\
C $n$=200 & 1447 & 461 & 26861(9) & 204 \\ \hline
        \end{tabular}
    \end{center}
    \caption{Average number of $f/g$ evaluations needed to reach $|f(\bu_i)-f^*|<10^{-6}$ for 10 instances of Problems A--C with random initial guess and with different sizes. Numbers in brackets give the number of random trials (out of 10) that did not converge to the required tolerance within 1500 iterations (if any).}
    \label{tab:ABC}
\end{table}
%------------------------------------------------------------
%---------------------------------------------------------------------------------------------------------------------------------
\begin{figure}[!htbp]
  \centering
  \scalebox{0.35}{
  \includegraphics{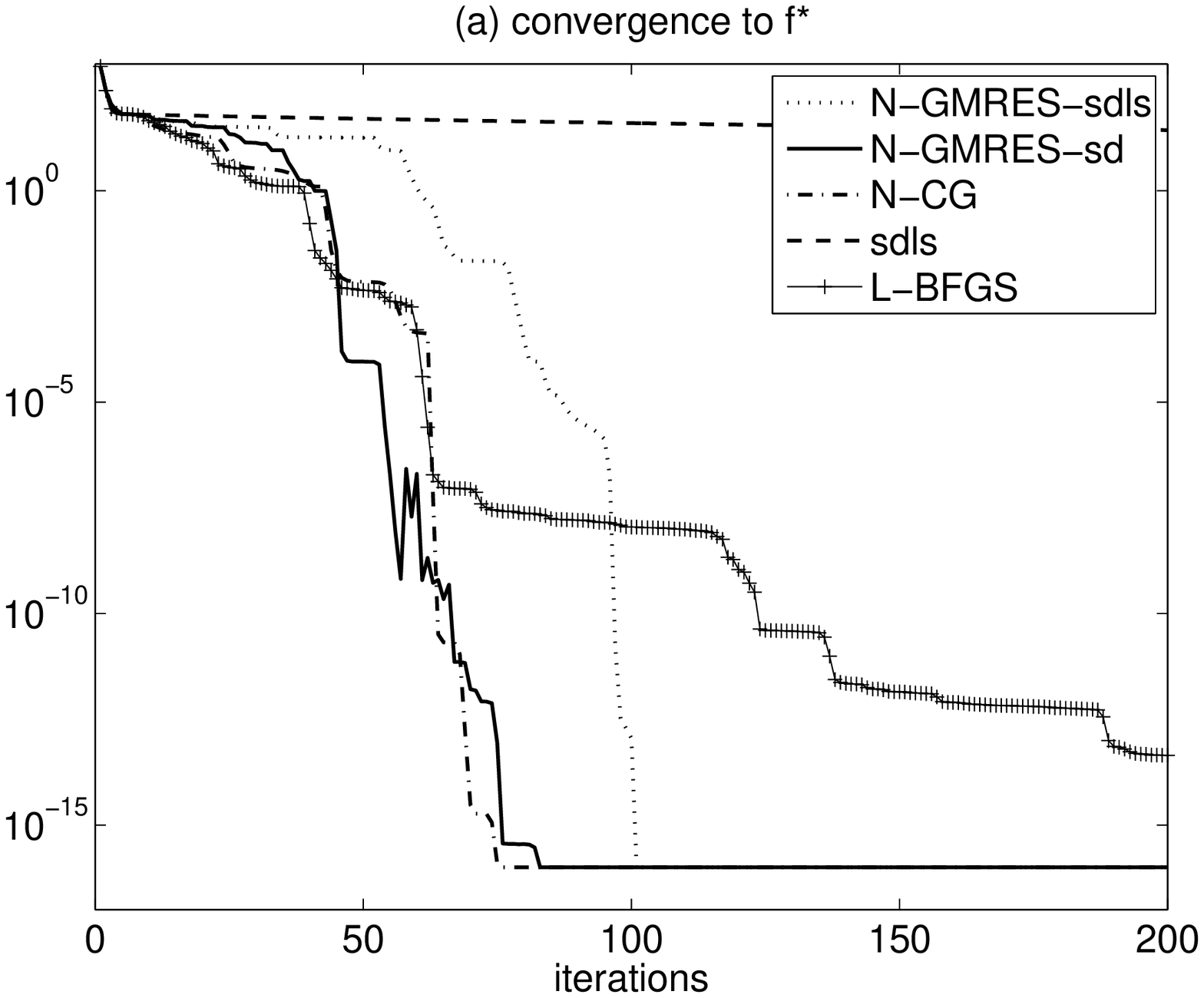}
  \includegraphics{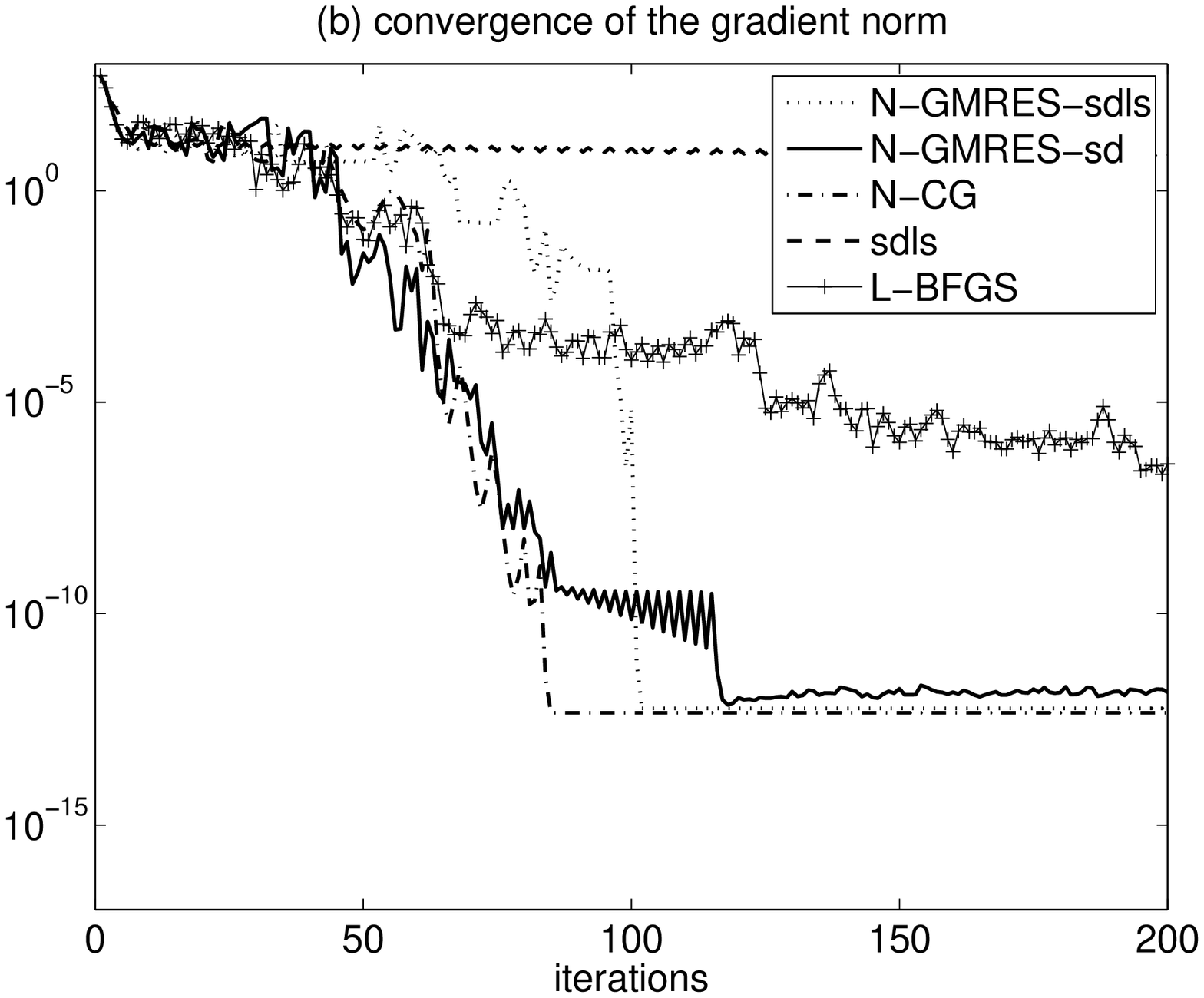}
  }
  \scalebox{0.35}{
  \includegraphics{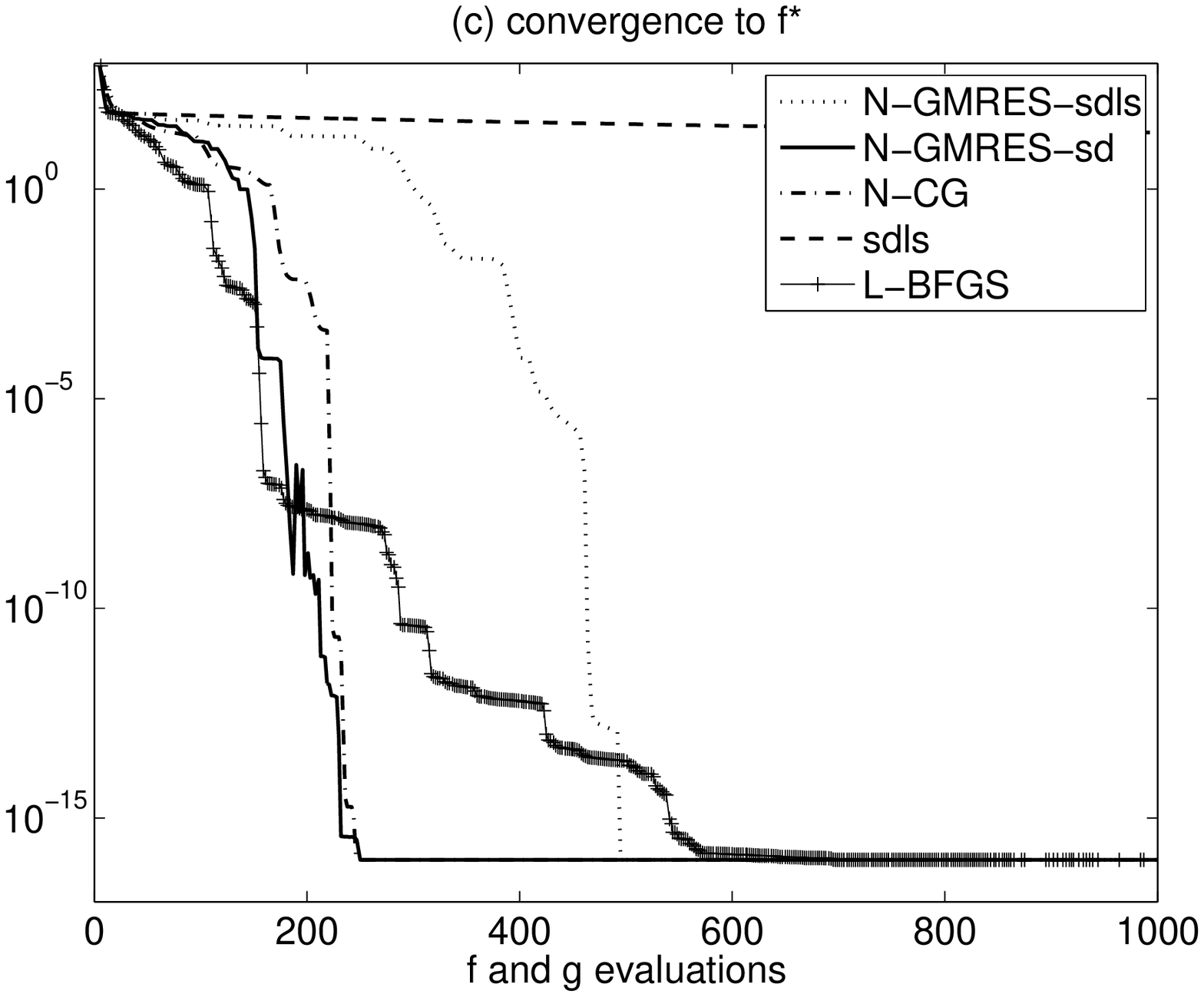}
  \includegraphics{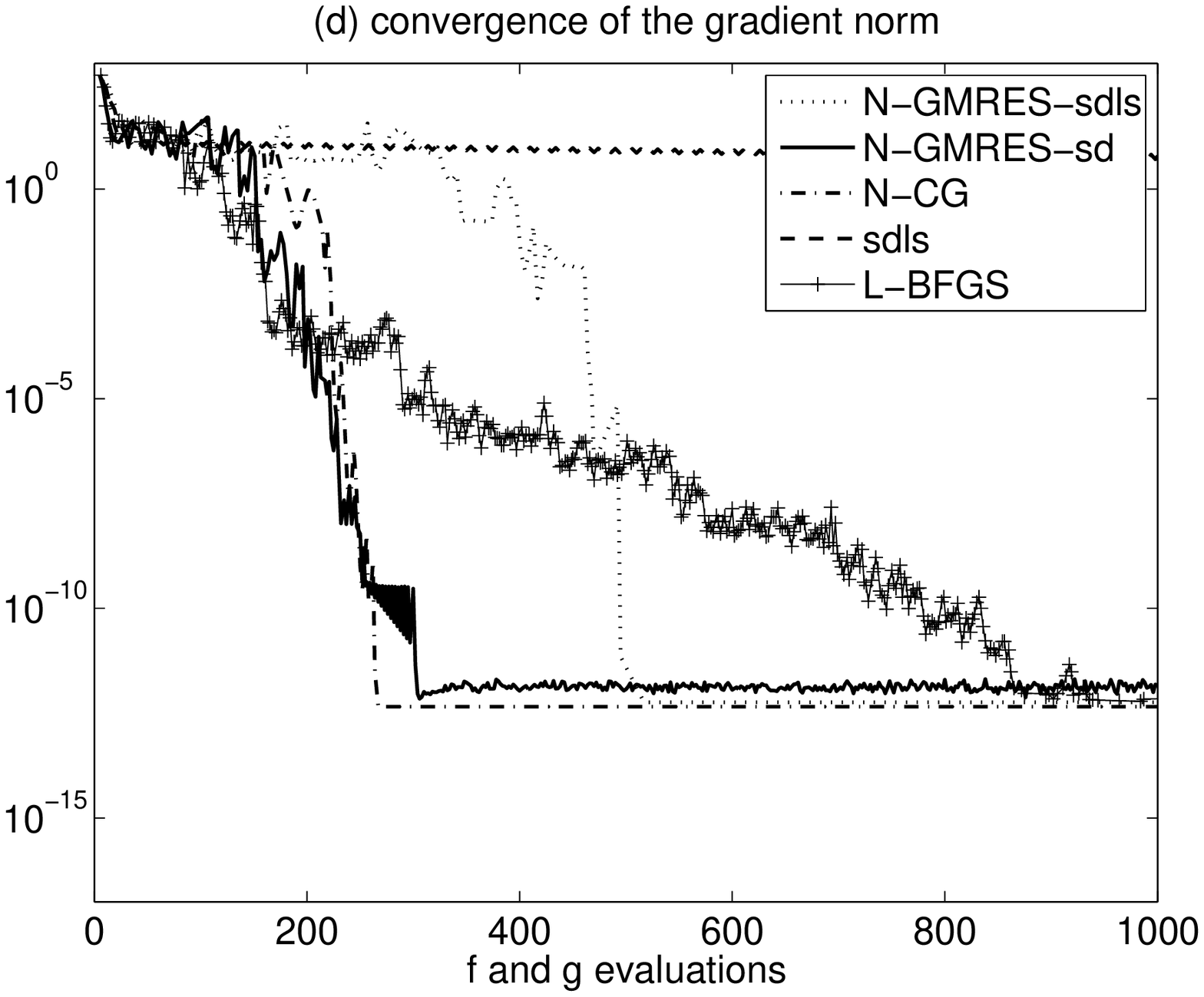}
  }
   \caption{Problem D ($n=1000$). Convergence comparison.}
   \label{fig:D}
\end{figure}    
%---------------------------------------------------------------------------------------------------------------------------------

Table \ref{tab:ABC} confirms the trends that were already present in the specific instances of test problems A--C that were shown in Figures \ref{fig:A}, \ref{fig:B} and \ref{fig:C}. The table gives the average number of $f/g$ evaluations that were needed to reach $|f(\bu_i)-f^*|<10^{-6}$ for 10 random instances of Problems A--C with different sizes. For Problems A and B, N-GMRES-sdls and N-GMRES-sd consistently give $f/g$ evaluation counts that are of the same order of magnitude as N-CG. N-GMRES-sd comes close to being competitive with N-CG. L-BFGS is the fastest method for all problems in Table \ref{tab:ABC}. For the more difficult Problem C, both N-GMRES-sdls, N-GMRES-sd and L-BFGS are significantly faster than N-CG, which appears to have convergence difficulties for this problem. N-GMRES-sd is clearly faster than N-GMRES-sdls for all tests.

%*****************************************************************
\subsection{Numerical Results for Problems D--G}
\label{subsec:DG}
%*****************************************************************
Figure \ref{fig:D} gives convergence plots for a single instance of Problem D. It confirms the observations from Figures \ref{fig:A}, \ref{fig:B} and \ref{fig:C}: for this standard test problem from \cite{MoreTest}, stand-alone sdls again is very slow, and N-GMRES-sdls and N-GMRES-sd significantly speed up steepest descent convergence. N-GMRES-sdls and N-GMRES-sd have iteration and $f/g$ counts that are of the same order of magnitude as N-CG and L-BFGS, and in particular N-GMRES-sd is competitive with N-CG and L-BFGS. Convergence plots for instances of Problems E--G show similar behaviour and are not presented.

%------------------------------------------------------------
\begin{table}[h!]
    \begin{center}
        \begin{tabular}{|l|c|c|c|c|}
\hline
problem  & N-GMRES-sdls & N-GMRES-sd & N-CG & L-BFGS\\
 \hline
D $n$=500 & 525 & 172 & 222 & 166 \\
D $n$=1000 & 445 & 211 & 223 & 170 \\
E $n$=100 & 294 & 259 & 243 & 358 \\
E $n$=200 & 317 & 243 & 240 & 394 \\
F $n$=200 & 140 & 102(1) & 102 & 92 \\
F $n$=500 & 206(1) & 175(1) & 135 & 118 \\
G $n$=100 & 1008(2) & 152 & 181 & 358 \\
G $n$=200 & 629(1) & 181 & 137 & 240 \\
%D $n$=500 & 525 & 172 & 222 \\
%D $n$=1000 & 445 & 211 & 223 \\
%E $n$=100 & 294 & 259 & 243 \\
%E $n$=200 & 317 & 243 & 240 \\
%F $n$=200 & 140 & 102(1) & 102 \\
%F $n$=500 & 206(1) & 175(1) & 135 \\
%G $n$=100 & 1008(2) & 152 & 181 \\
%G $n$=200 & 629(1) & 181 & 137 \\
 \hline
        \end{tabular}
    \end{center}
    \caption{Average number of $f/g$ evaluations needed to reach $|f(\bu_i)-f^*|<10^{-6}$ for 10 instances of Problems D--G with random initial guess and with different sizes. Numbers in brackets give the number of random trials (out of 10) that did not converge to the required tolerance within 500 iterations (if any).}
    \label{tab:DEFG}
\end{table}
%------------------------------------------------------------
Table \ref{tab:DEFG} on $f/g$ evaluation counts for Problems E--G again confirms the trends that were observed before. N-GMRES-sdls and N-GMRES-sd give $f/g$ evaluation counts that are of the same order of magnitude as N-CG and L-BFGS, and N-GMRES-sd in particular is competitive with N-CG and L-BFGS.

%%%%%%%%%%%%%%%%%%%%%%%%%%%%%%%%%%%%%%%%%%%%%%%%%%%%%%%%%%
%%%%%%%%%%%%%%%%%%%%%%%%%%%%%%%%%%%%%%%%%%%%%%%%%%%%%%%%%%
\section{Conclusion}
\label{sec:conc}
%%%%%%%%%%%%%%%%%%%%%%%%%%%%%%%%%%%%%%%%%%%%%%%%%%%%%%%%%%
%%%%%%%%%%%%%%%%%%%%%%%%%%%%%%%%%%%%%%%%%%%%%%%%%%%%%%%%%%
%---------------------------------------------------------------------------------------------------------------------------------
\begin{figure}[!htbp]
  \centering
  \scalebox{0.5}{
  \includegraphics{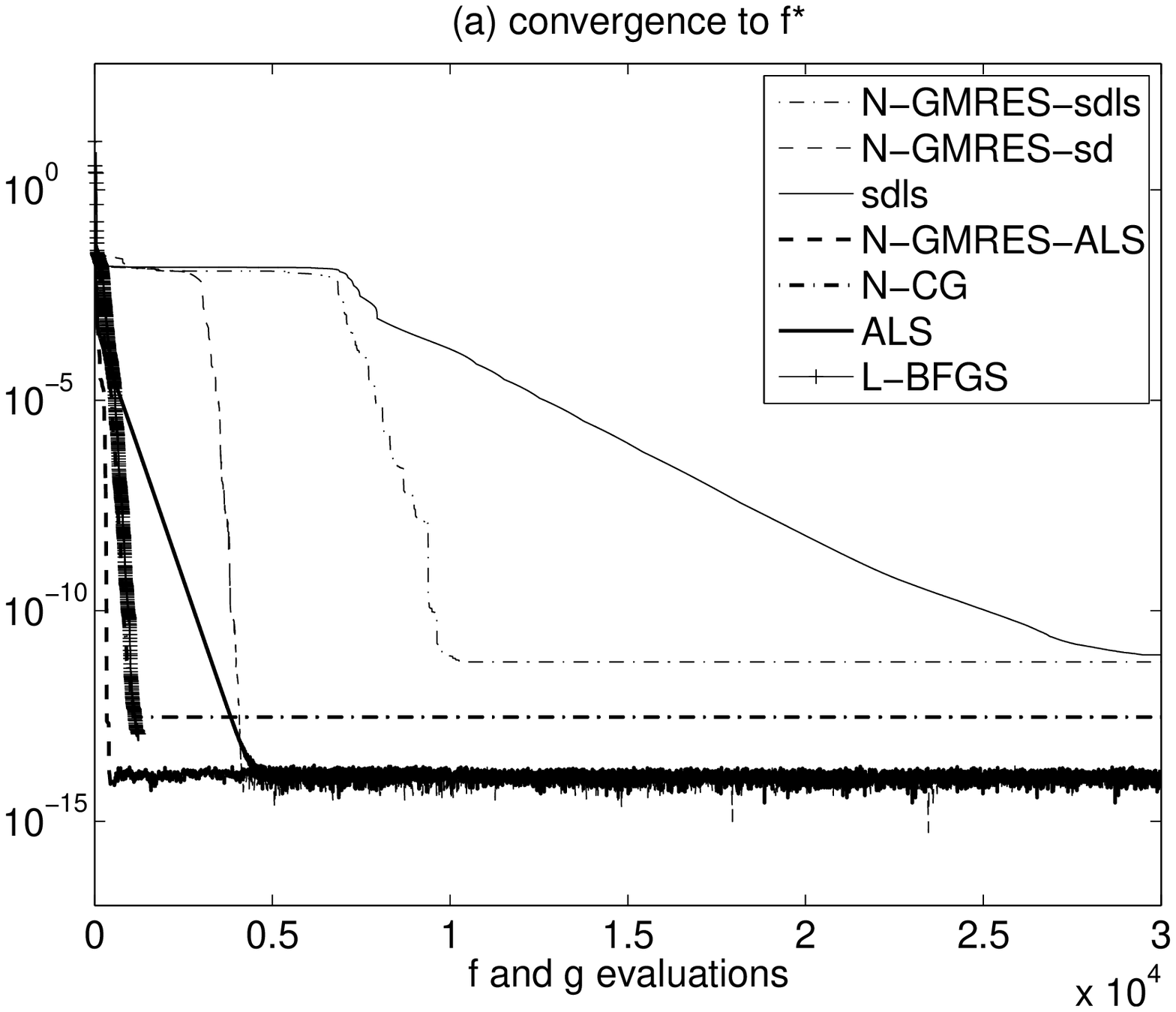}}
  \scalebox{0.5}{  
  \includegraphics{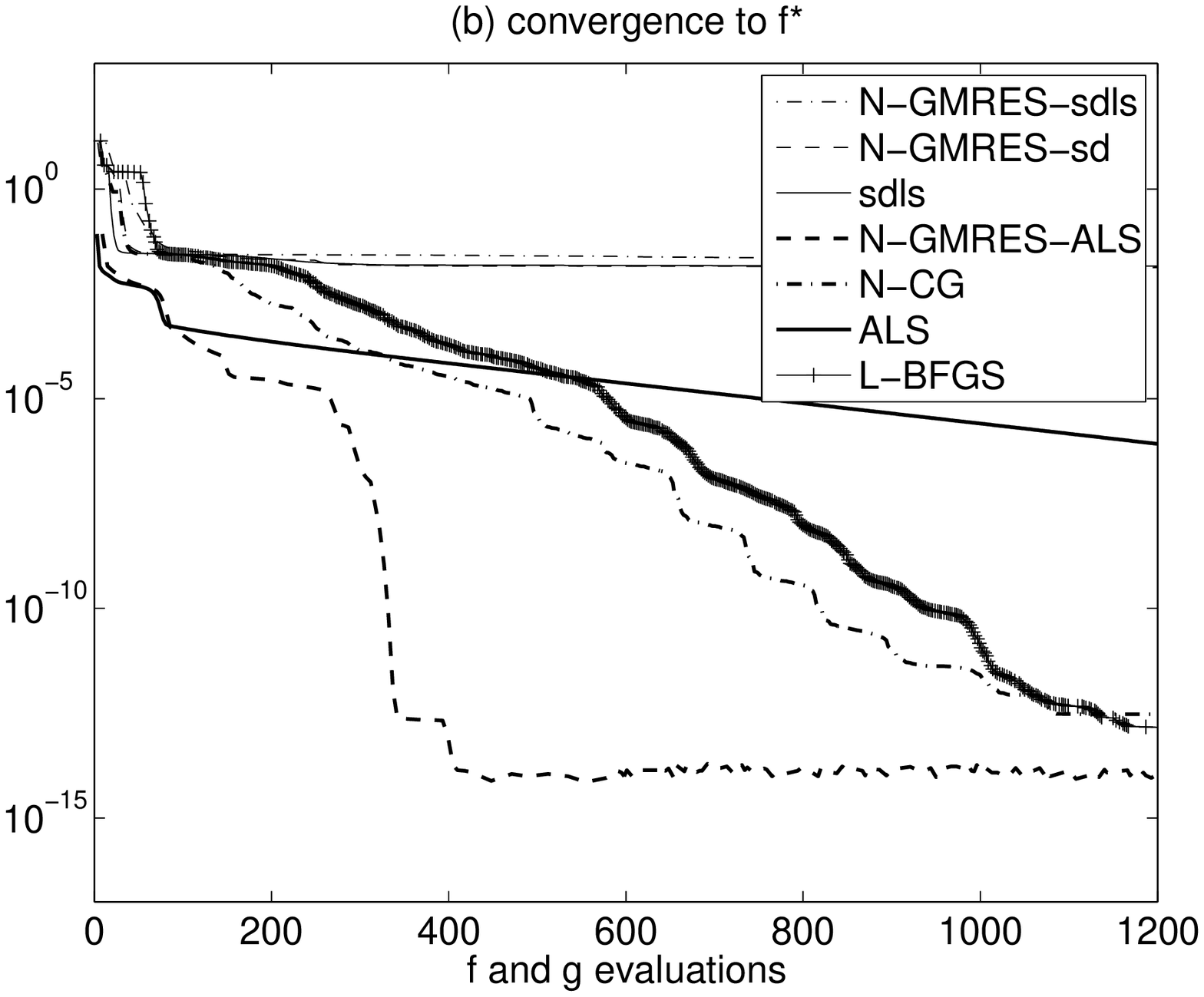}
  }
   \caption{Convergence histories of the 10-logarithm of $|f(\bu_i)-f^*|$ as a function
   of $f/g$ evaluations, for the canonical tensor approximation problem of Figures 1.2 and 1.3 in
   \cite{NGMRES}.
   Panel (a) shows that stand-alone sdls is very slow for this problem, and N-GMRES-sdls and N-GMRES-sd significantly speed up steepest descent. However, for this difficult problem, it is beneficial to use a more powerful nonlinear preconditioner. Using the ALS preconditioner in stand-alone fashion already provides faster convergence than N-GMRES-sdls and N-GMRES-sd. The zoomed view in Panel (b) shows that N-CG and L-BFGS are faster than stand-alone ALS when high accuracy is required, but N-GMRES preconditioned with the powerful ALS preconditioner is the fastest method by far, beating N-CG and L-BFGS by a factor of 2 to 3. This illustrates that the real power of the N-GMRES optimization algorithm may lie in its ability to employ powerful problem-dependent nonlinear preconditioners (ALS in this case).}
   \label{fig:CP}
\end{figure}    
%---------------------------------------------------------------------------------------------------------------------------------

In this paper, we have proposed and studied steepest descent preconditioning as a universal preconditioning approach
for the N-GMRES optimization algorithm that we recently introduced in the context of a canonical tensor approximation problem and ALS preconditioning \cite{NGMRES} (Paper I).
We have considered two steepest descent preconditioning process variants, one with a line search, and the other one with a predefined step length.
The first variant is significant because we showed that it leads to a globally convergent optimization method, but the second variant proved more efficient in numerical tests, with no apparent degradation in convergence robustness.
Numerical tests showed that the two steepest-descent preconditioned N-GMRES methods both speed up stand-alone steepest descent optimization very significantly, and are competitive with standard N-CG and L-BFGS methods, for a variety of test problems.
These results serve to theoretically and numerically establish steepest-descent preconditioned N-GMRES as a general optimization method for unconstrained nonlinear optimization, with performance that appears promising compared to established techniques.

However, we would like to argue that the real potential of the N-GMRES optimization framework lies in the fact
that it can use problem-dependent nonlinear preconditioners that are more powerful than steepest descent.
Preconditioning of N-CG in the form of (linear) variable transformations is an area of active research \cite{HagerPrecond}. However, it is interesting to note that our N-GMRES optimization framework naturally allows for a more general type of preconditioning: any nonlinear optimization process $M(.)$ can potentially be used as a nonlinear preconditioner in the framework, or, equivalently, N-GMRES can be used as a simple wrapper around any other iterative optimization process $M(.)$ to seek acceleration of that process.
This can be illustrated with the following example, in which we first apply N-GMRES with the steepest descent preconditioners proposed in this paper, to a canonical tensor approximation problem from \cite{NGMRES}.
(In particular, we consider the canonical tensor approximation problem of Figures 1.2 and 1.3 in \cite{NGMRES}, in which a rank-three canonical tensor approximation (with 450 variables) is sought for a three-way data tensor of size $50\times50\times50$.) Panel (a) of Fig.\ \ref{fig:CP} shows how stand-alone steepest descent (sdls) is very slow for this problem: it requires more than 30,000 $f/g$ evaluations. (The tensor calculations are performed in matlab using the Tensor Toolbox \cite{KoldaTOOLBOX}. For this problem, we use $\delta=10^{-3}$ in (\ref{eq:steepestB}).) The GMRES-sdls and N-GMRES-sd convergence profiles confirm once more one of the main messages of this paper: steepest-descent preconditioned N-GMRES speeds up stand-alone steepest descent very significantly. However, steepest descent preconditioning (which we have argued is in some sense equivalent to non-preconditioned GMRES for linear systems) is not powerful enough for this difficult problem, and a more advanced preconditioner is required. Indeed, Panel (a) of Fig.\ \ref{fig:CP} shows that the stand-alone ALS process is already more efficient than steepest-descent preconditioned N-GMRES. Panel (b) indicates, however, that N-GMRES preconditioned by ALS is a very effective method for this problem: it speeds up ALS very signficantly, and is much faster than N-CG and L-BFGS, by a factor of 2 to 3. (Panel (b) of Fig.\ \ref{fig:CP} illustrates the findings from extensive tests comparing ALS, N-CG and ALS-preconditioned N-GMRES that were reported in Paper I and \cite{AcarCPOPT}.)

In the case of GMRES for linear systems, non-preconditioned GMRES (or: GMRES with the identity preconditioner) is often just a starting point. For many difficult problems it converges too slowly, and there is a very extensive and ever expanding research literature on developing advanced problem-dependent preconditioners that in many cases speed up convergence very significantly. In the same way, the present paper is likely not more than a starting point in theoretically and numerically establishing the N-GMRES optimization method with general steepest descent preconditioning process. As the results shown in Fig.\ \ref{fig:CP} already indicate, we expect that the real power of the N-GMRES optimization framework will turn out to lie in its ability to use powerful problem-dependent nonlinear preconditioners. This suggests that further exploring N-GMRES optimization with advanced preconditioners may lead to efficient numerical methods for a variety of nonlinear optimization problems.

\section*{Acknowledgments}
This work was sponsored by the Natural Sciences and Engineering Research Council of Canada and by Lawrence Livermore National Laboratory under subcontract B594099. The research was conducted during a sabbatical visit at the Algorithms and Complexity Department of the Max Planck Institute for Informatics in Saarbruecken, whose hospitality is greatly acknowledged.
%%%%%%%%%%%%%%%%%%%%%%%%%%%%%%%%%%%%%%%%%%%%%%%%%%%%%%%%%%
%%%%%%%%%%%%%%%%%%%%%%%%%%%%%%%%%%%%%%%%%%%%%%%%%%%%%%%%%%
%%%%%%%%%%%%%%%%%%%%%%%%%%%%%%%%%%%%%%%%%%%%%%%%%%%%%%%%%%
%%%%%%%%%%%%%%%%%%%%%%%%%%%%%%%%%%%%%%%%%%%%%%%%%%%%%%%%%%
%%%%%%%%%%%%%%%%%%%%%%%%%%%%%%%%%%%%%%%%%%%%%%%%%%%%%%%%%%
%%%%%%%%%%%%%%%%%%%%%%%%%%%%%%%%%%%%%%%%%%%%%%%%%%%%%%%%%%
%%%%%%%%%%%%%%%%%%%%%%%%%%%%%%%%%%%%%%%%%%%%%%%%%%%%%%%%
% Bibliography
%%%%%%%%%%%%%%%%%%%%%%%%%%%%%%%%%%%%%%%%%%%%%%%%%%%%%%%%
%\nocite{*}
%\bibliographystyle{siam}
%\bibliography{lsnonlin}

\begin{thebibliography}{10}

\bibitem{AcarCPOPT}{\sc 
E. Acar, D.M. Dunlavy, and T.G. Kolda}, {\em A Scalable Optimization Approach for Fitting Canonical Tensor Decompositions},  Journal of Chemometrics, 25 (2011), pp. 67--86.

\bibitem{KoldaTOOLBOX}{\sc 
B.W. Bader and T.G. Kolda}, {\em MATLAB Tensor Toolbox Version 2.4}, http://csmr.ca.sandia.gov/~tgkolda/TensorToolbox/, March 2010.

\bibitem{NGMRES} {\sc H. De Sterck}, {\em A Nonlinear GMRES Optimization Algorithm for Canonical Tensor Decomposition}, submitted to SIAM J. Sci. Comp., 2011, arXiv:1105.5331.

\bibitem{POBLANO}{\sc 
D.M. Dunlavy, T.G. Kolda, and E. Acar}, {\em Poblano v1.0: A Matlab Toolbox for Gradient-Based Optimization}, Technical Report SAND2010-1422, Sandia National Laboratories, Albuquerque, NM and Livermore, CA, March 2010.

\bibitem{SaadAnderson}{\sc H. Fang and Y. Saad}, {\em Two classes of multisecant methods for nonlinear acceleration}, Numerical Linear Algebra with Applications, 16 (2009), pp. 197--221.

\bibitem{NocedalNCG}{\sc J.C. Gilbert and J. Nocedal}, {\em Global Convergence Properties of Conjugate Gradient Methods for Optimization}, SIAM J. Optim., 2 (1992), pp. 21--42.

\bibitem{HagerPrecond}{\sc W.W. Hager and H. Zhang}, {\em A Survey of Nonlinear Conjugate Gradient Methods}, Pacific Journal of Optimization, 2 (2006), pp. 35--58.

\bibitem{MoreThuente}{\sc 
J.J. Mor\'{e} and D.J. Thuente}, {\em Line search algorithms with guaranteed sufficient decrease}, ACM Transactions on Mathematical Software, 20 (1994), pp. 286--307.

\bibitem{MoreTest}{\sc J.J. Mor\'{e}, B.S. Garbow, and K.E. Hillstrom}, {\em Testing Unconstrained Optimization Software},
 ACM Trans. Math. Softw., 7 (1981), pp. 17--41.
 
\bibitem{Nocedal}
{\sc J. Nocedal and S.J. Wright}, {\em Numerical optimization}, Second Edition, Springer, Berlin, 2006.

\bibitem{OosterleeNGMRES-ETNA}{\sc C.W. Oosterlee}, {\em On multigrid for linear complementarity problems with application to American-style options}, 
Electronic Transactions on Numerical Analysis, 15 (2003), pp. 165--185.

\bibitem{OosterleeNGMRES-SISC}{\sc C.W. Oosterlee and T. Washio}, {\em Krylov Subspace Acceleration of Nonlinear Multigrid with Application to Recirculating Flows}, SIAM J. Sci. Comput., 21 (2000), pp. 1670--1690.

\bibitem{SaadFlexible}{\sc Y. Saad}, {\em A flexible inner-outer preconditioned GMRES algorithm},
SIAM J. Sci. Comp., 14 (1993), pp. 461--469. 

\bibitem{SaadBook}{\sc Y. Saad}, {\em Iterative Methods for Sparse Linear Systems}, Second Edition, SIAM, Philadelphia, 2003.

\bibitem{SaadGMRES}{\sc Y. Saad and M.H. Schultz}, {\em GMRES: A generalized minimal residual algorithm for solving nonsymmetric linear systems}, SIAM J. Sci. Comp., 7 (1986), pp. 856--869.

\bibitem{RRE}{\sc D.A. Smith, W.F. Ford, and A. Sidi}, {\em Extrapolation methods for vector sequences}, SIAM Rev., 29 (1987), pp. 199--234.

\bibitem{Walker}{\sc H. Walker and P. Ni}, {\em Anderson acceleration for fixed-point iterations}, to appear in SIAM J. Numer. Anal (2011).

\bibitem{WashioNGMRES-ETNA}{\sc 
T. Washio and C.W. Oosterlee}, {\em Krylov subspace acceleration for nonlinear multigrid schemes}, 
Electronic Transactions on Numerical Analysis,  6 (1997), pp. 271--290.

%\bibitem{CANDECOMP}{\sc 
%J.D. Carroll and J.J. Chang}, {\em Analysis of individual differences in multidimensional scaling via an N-way generalization of ÒEckart-YoungÓ decomposition}, Psychometrika, 35 (1970), pp. 283--319.

%\bibitem{LathauwerSimul}{\sc 
%L. De Lathauwer}, {\em A link between the canonical decomposition in multilinear algebra and simultaneous matrix diagonalization}, SIAM Journal on Matrix Analysis and Applications, 28 (2006), pp. 642--666.

%\bibitem{LathauwerSchur}{\sc 
%L. De Lathauwer, B. De Moor, and J. Vandewalle}, {\em Computation of the canonical decomposition by means of a simultaneous generalized Schur decomposition}, SIAM Journal on Matrix Analysis and Applications, 26 (2004), pp. 295--327.

%\bibitem{amg-theory}{\sc A.~Brandt}, {\em Algebraic multigrid theory: 
%The symmetric case}, Appl. Math. Comp. 19:23-56, 1986.
%\bibitem{robustAMG} {\sc A. Cleary, R. Falgout, V. Henson, J. Jones, T. Manteuffel, S. McCormick, G. Miranda, and J. Ruge}, {\em Robustness and algorithmic scalability of algebraic multigrid (AMG)}, SIAM J. Sci. Comp. 21:1886-1908, 2000.
%----------------------------------
%\bibitem{PARAFAC}{\sc 
%R.A. Harshman}, {\em Foundations of the PARAFAC procedure: Models and conditions for an ÒexplanatoryÓ multi-modal factor analysis}, UCLA working papers in phonetics, 16 (1970), pp. 1--84.

%\bibitem{Ishteva}{\sc 
%M. Ishteva, L. De Lathauwer, P. Absil, and S. Van Huffel}, {\em Differential-geometric Newton method for the best rank-($R_1$
%,$R_2$,$R_3$) approximation of tensors}, Numerical Algorithms, 51 (2009), pp. 179--194.

%\bibitem{ComonLS}{\sc 
%M. Rajih, P. Comon, and R.A. Harshman}, {\em Enhanced line search: A novel method to accelerate {PARAFAC}}, SIAM Journal on Matrix Analysis and Applications, 30 (2008), pp. 1128--1147.

%\bibitem{Savas}{\sc 
%B. Savas and L. Eld\'{e}n}, {\em Krylov-Type Methods for Tensor Computations}, submitted (2010), arXiv:1005.0683v2.

%\bibitem{TomasiPARAFAC}{\sc
%G. Tomasi and R. Bro}, {\em A comparison of algorithms for fitting the PARAFAC model}, Computational Statistics and Data Analysis, 50 (2006), pp. 1700--1734.

%\bibitem{KoldaSIREV}{\sc 
%T.G. Kolda and B.W. Bader}, {\em Tensor Decompositions and Applications}, SIAM Review, 51 (2009), pp. 455--500.

%-------------------------


\end{thebibliography}
%\clearpage

%%%%%%%%%%%%%%%%%%%%%%%%%%%%%%%%%%%%%%%%%%%
%%%%%%%%%%%%%%%%%%%%%%%%%%%%%%%%%%%%%%%%%%%
%%%%%%%%%%%%%%%%%%%%%%%%%%%%%%%%%%%%%%%%%%%

\end{document}